\documentclass[a4paper,12pt]{amsart}

\usepackage{graphicx}
\usepackage{tikz-cd}
\usepackage{amssymb}
\usepackage{upquote}
\usepackage{mathtools}
\usepackage{amsthm}
\usepackage[normalem]{ulem}
\usetikzlibrary{arrows}
\usepackage{color}
\usepackage{multirow}
\usepackage{enumitem}
\usepackage{textcomp}
\usepackage{xfrac}
\usepackage{hyperref, aliascnt}
\usepackage{extarrows}
\usepackage[capitalise, nameinlink, noabbrev, nosort]{cleveref}
\usepackage{extarrows}

\hypersetup{pdfauthor={Kristin Courtney, Wilhelm Winter}, pdftitle={Nuclearity and c.p.\ inductive systems}, colorlinks=true , linkcolor=darkgray, linkbordercolor=darkgray, citecolor=darkgray, citebordercolor=darkgray, linktocpage=true}

\newtheorem{theorem}{Theorem}[section]
\newaliascnt{lemma}{theorem}

\aliascntresetthe{lemma}

\newaliascnt{proposition}{theorem}
\newtheorem{proposition}[proposition]{Proposition}
\aliascntresetthe{proposition}

\newaliascnt{corollary}{theorem}
\newtheorem{corollary}[corollary]{Corollary}
\aliascntresetthe{corollary}

\newaliascnt{claim}{theorem}

\aliascntresetthe{claim}

\newtheorem*{theorem*}{Theorem}

\theoremstyle{definition}

\newaliascnt{definition}{theorem}
\newtheorem{definition}[definition]{Definition}
\aliascntresetthe{definition}

\newtheorem*{definition*}{Definition}

\newaliascnt{notation}{theorem}

\aliascntresetthe{notation}

\newaliascnt{remark}{theorem}
\newtheorem{remark}[remark]{Remark}
\aliascntresetthe{remark}

\newaliascnt{remarks}{theorem}
\newtheorem{remarks}[remarks]{Remarks}
\aliascntresetthe{remarks}

\newaliascnt{example}{theorem}
\newtheorem{example}[example]{Example}
\aliascntresetthe{example}

\newaliascnt{examples}{theorem}
\newtheorem{examples}[examples]{Examples}
\aliascntresetthe{examples}

\newaliascnt{question}{theorem}

\aliascntresetthe{question}

\newtheorem*{corollary*}{Corollary}
\newtheorem*{proposition*}{Proposition}
\newtheorem*{remark*}{Remark}

\newtheorem*{example*}{Example}
\numberwithin{equation}{section}

\newcommand{\sbt}{\,\begin{picture}(-1,1)(-1,-2.5)\circle*{3}\end{picture}\,\,\, }

\allowdisplaybreaks

\begin{document}
\def\C{\mathbb{C}}
\def\R{\mathbb{R}}
\def\T{\mathbb{T}}
\def\N{\mathbb{N}}
\def\M{\mathrm{M}}
\def\SS{\mathcal{S}}
\def\id{\mathrm{id}}
\def\Cstar{\mathrm{C}^*}
\def\Span{\textrm{span}}
\def\CPCstar{\mathrm{CPC}^*}
\def\Limfr{\overset{\tikz \draw [arrows = {->[harpoon]}] (-1,0) -- (0,0);}{(F_n,\rho_n)}_n}
\def\Limft{\overset{\tikz \draw [arrows = {->[harpoon]}] (-1,0) -- (0,0);}{(F_n,\theta_n)}_n}
\def\Lima{\overset{\tikz \draw [arrows = {->[harpoon]}] (-1,0) -- (0,0);}{(A_n,\rho_n)}_n}
\def\Limbr{\overset{\tikz \draw [arrows = {->[harpoon]}] (-1,0) -- (0,0);}{(B_n,\rho_n)}_n}
\def\Limbt{\overset{\tikz \draw [arrows = {->[harpoon]}] (-1,0) -- (0,0);}{(B_n,\theta_n)}_n}
\def\Limbg{\overset{\tikz \draw [arrows = {->[harpoon]}] (-1,0) -- (0,0);}{(B_n,\gamma_n)}_n}
\def\Re{\mathrm{Re}}
\def\Im{\mathrm{Im}}
\def\e{\tiny{\textup{\textcircled{e}}}}
\def\h{\tiny{\textup{\textcircled{h}}}}
\def\Lim{\overset{\tikz \draw [arrows = {->[harpoon]}] (-1,0) -- (0,0);}{(F_n,\rho_n)}_n}

\title{
Images of Order Zero Maps 
%\\ 
%Inductive Systems and Nuclear Dimension
}
\author[K.\ Courtney]{Kristin Courtney}
     \address{Department of Mathematics and Computer Science, University of Southern Denmark, Campusvej 55,
5230 Odense M, Denmark}
     \email{kcourtney@imada.sdu.dk}
      \author[W.\ Winter]{Wilhelm Winter}
     \address{Mathematical Institute, WWU M\"unster, Einsteinstr.\ 62, 48149 M\"unster, Germany}
     \email{wwinter@uni-muenster.de}
     \subjclass[2010]{46L05} %: General theory of $C^*$-algebras
% 22A22%: Topological groupoids (including differentiable and Lie groupoids)
% }
 \keywords{Completely positive approximation, nuclearity, inductive limits, order zero maps, completely positive maps, NF systems, NF algebras}
 \thanks{
This research was supported by the Deutsche Forschungsgemeinschaft (DFG, German Research Foundation) under Germany's Excellence Strategy -- EXC 2044 -- 390685587, Mathematics M\"unster -- Dynamics -- Geometry -- Structure, the Deutsche Forschungsgemeinschaft (DFG, German Research Foundation) –- Project-ID 427320536 –- SFB 1442, and ERC Advanced Grant 834267 -- AMAREC}
\dedicatory{Dedicated to the memory of Eberhard Kirchberg}

\date{\today}

\begin{abstract}
We give sufficient conditions allowing one to build a $\Cstar$-algebraic structure on a self-adjoint linear subspace of a $\Cstar$-algebra in such a way that the subspace is naturally identified with the resulting $\Cstar$-algebra via a completely positive order zero map. This leads to necessary and sufficient conditions to realize a self-adjoint linear subspace of a $\Cstar$-algebra as the image of an order zero map from a $\Cstar$-algebra and to determine whether it is completely order isomorphic to a $\Cstar$-algebra via an order zero map. %We also consider the unitization of this $\Cstar$-algebra in terms of the initial space.
\end{abstract}

\maketitle

\tableofcontents

\section*{Introduction}
\renewcommand*{\thetheorem}{\Alph{theorem}}
\renewcommand*{\thedefinition}{\Alph{definition}}
\renewcommand*{\theproposition}{\Alph{proposition}}
\renewcommand*{\thecorollary}{\Alph{theorem}}

\noindent Completely positive order zero maps between $\Cstar$-algebras are those which preserve the $^*$-linear and matrix order structures as well as orthogonality. They were introduced in this form in \cite{WZ09},  following ideas of \cite{Win09} and \cite{KW} (a collaboration of the second named author  with Kirchberg) in the case of finite-dimensional domains, and the pioneering work  \cite{Wol94} of Wolff in a both more special and more general context. Order zero maps play a vital role in the structure and classification theory of nuclear $\Cstar$-algebras and in approximation properties related to $\Cstar$-dynamical systems, and they are a highly effective tool when passing between $\Cstar$-algebraic and von Neumann algebraic regularity properties; see  \cite{BGSS, CETWW, CSSWW, Sato2019, SSW, TWW, W12, WZ10} for examples of applications of order zero maps.

In \cite{CW1}, %and the forthcoming \cite{CW3}, 
we take inspiration from the work of Blackadar and Kirchberg (cf.~ \cite{BK97,BK99,BK01,BK11}) and develop notions of inductive systems with completely positive asymptotically order zero maps which allow for inductive limit characterizations of nuclearity and of finite nuclear dimension for all separable $\Cstar$-algebras. The proofs and constructions in the first article were heavily informed by the systematic study in the present paper, and the second directly depends on it. However, applications of the results here are not limited to inductive limit constructions. One notable upshot is a characterization of when a self-adjoint linear subspace of a $\Cstar$-algebra is completely order isomorphic to a $\Cstar$-algebra \emph{via an order zero map} (\cref{Cor E}). This answers, in a special case, the question of how to detect when a self-adjoint linear subspace of a $\Cstar$-algebra is completely order isomorphic to a $\Cstar$-algebra. The order zero condition is natural in this context because it occurs in abundance (in particular for nuclear $\Cstar$-algebras; c.~f.~ \cite{CW1}), and since it yields an easy characterization of complete order embeddings which is not available for general completely positive maps (c.~f.~ \cref{a: isometric unitization} and \cref{Bestellung} below). 

\bigskip

As a consequence of the structure theorem for order zero maps, when the domain of a completely positive map is a unital $\Cstar$-algebra $A$, the map is order zero exactly when the following identity holds for all elements $a,b\in A$:
\begin{align*}\label{dager}
\theta(1) \theta(ab) = \theta(a)\theta(b).\tag{a}
\end{align*}
From this one directly obtains two necessary conditions for a space to be of the form $\theta(A)$, i.e., an image of a c.p.\ order zero map from a unital $\Cstar$-algebra: First, $\theta(1_A)$ commutes with the image, $\theta(1_A)\in \theta(A)'$, and second $\theta(A)^2=\theta(1_A)\theta(A)$. 

Our main result in this article (\cref{theorem A}) draws inspiration from the characterization \eqref{dager} to provide sufficient conditions allowing one to realize a self-adjoint linear subspace of a $\Cstar$-algebra as a (pre-)$\Cstar$-algebra (albeit likely not a sub-$\Cstar$-algebra of the ambient $\Cstar$-algebra). One of these conditions, an adaptation of the notion of a local order unit from the theory of abstract operator systems (as described in \cite[Chapter 13]{Pau02}), requires a brief introduction. It is necessary since we do not want to restrict ourselves to unital domain $\Cstar$-algebras.

\begin{definition}\label{def 1}
Let $B$ be a $\Cstar$-algebra and $X\subset B$ a self-adjoint linear subspace. We call an element $e\in B_+^1$ a \emph{local order unit} for $X$ if, for any $x\in X_{\mathrm{s.a.}}$, there exists an $R>0$ so that $Re\geq x$. We call $e$ an \emph{        order unit} for $X$ if there exists an $R>0$ so that $R\|x\|e\geq x$ for every $x\in X_{\mathrm{s.a.}}$. For an        order unit, the minimal\footnote{See Remark \ref{rmk: Order scale at least 1}(ii) for a proof that the scaling factor is in fact a minimum.} such $R$ is called the {\em scaling factor} of $e$. 
\end{definition}
The unit of a $\Cstar$-algebra is an        order unit with scaling factor $1$ for any self-adjoint linear subspace. The image of the unit under a completely positive map into another $\Cstar$-algebra is still a local order unit for the image of the domain, and when the map is bounded below it remains an        order unit, with scaling factor dependent on the norm of the inverse of the map. 

Now we are ready to state our first main result, which appears as \cref{theorem57} below.

\begin{theorem}\label{theorem A}
Let $B$ be a $\Cstar$-algebra and $X\subset B$ a self-adjoint linear subspace. Suppose there exists a local order unit $e\in B_+^1\cap X'$ for $X$ so that $X^2\subset eX$. 
Then there is an associative bilinear map ${\sbt}\colon X\times X\longrightarrow X$ satisfying
\begin{align}\label{mult id''}
    xy=e(x{\sbt} y)\tag{b}
\end{align}
for all $x,y\in X$, so that $(X,{\sbt})$ is a $^*$-algebra, with unit $e$ when $e\in X$, and there is a norm $\|\cdot\|_{\sbt}\colon X\longrightarrow [0,\infty)$ 
with respect to which $(X,{\sbt})$ is a pre-$\Cstar$-algebra. 
\end{theorem}

It follows that the completion of the $^*$-algebra $(X,\sbt)$ with respect to ${\|\cdot\|_{\sbt}}$ is a $\Cstar$-algebra, which we call the \emph{$\Cstar$-algebra associated to $(X,e)$} and denote with $\Cstar_{\sbt}(X)$. (We will make sure that suppressing the $e$ from the notation results in no ambiguity).  From the structure theorem for order zero maps (\cite[Theorem 3.3]{WZ09}) we know that an example of a self-adjoint linear subspace of a $\Cstar$-algebra $B$ satisfying the assumptions of \cref{theorem A} is the image of a c.p.c.\ order zero map $\theta\colon A\longrightarrow  B$ from a unital $\Cstar$-algebra (with $\theta(1_A)=e$). Another example comes from \cite[Section 2]{CW1}. In \cite[Proposition 2.6]{CW1}, the role of $B$ was played by $F_\infty$, $X$ was $\Lim$, and the element $[(\rho_{n+1,n}(1_{F_n}))_n]\in (F_\infty)_+^1$ was suggestively denoted by $e$. In this setting, \cite[Lemma 2.5(i) and (ii)]{CW1} established that $e\in X'$ is an        order unit for $X$ with scaling factor $1$. (In \cite[Section 2]{CW1}, the norm $\|\cdot\|_{\sbt}$ was simply the norm on $F_\infty$. We will shortly see why a scaling factor $1$ guarantees that the two norms coincide.)

If $X$ is already complete with respect to $\|\cdot\|_{\sbt}$, then $\Cstar_{\sbt}(X)$ and $X$ are equal as $^*$-vector spaces. This happens for instance when $X$ is the image of an order zero map $\theta\colon A\longrightarrow B$ from a unital $\Cstar$-algebra with $e=\theta(1_A)$ (see \cref{corollary D}) and also when $X$ is closed in $B$. (Indeed, we shall see that $X$ is closed in $B$ if and only if it is closed with respect to $\|\cdot\|_{\sbt}$ and $e$ is an        order unit for $X$, cf.\ \cref{closed uniform order unit}.) Though many of our results hold outside the setting where $X$ is complete with respect to $\|\cdot\|_{\sbt}$, our main applications and examples satisfy this assumption, and so we mostly adopt it for the current exposition.

As mentioned above, the image of a c.p.c.\ order zero map from a unital $\Cstar$-algebra satisfies the assumptions of \cref{theorem A}. 
A key feature of our study is the converse. 

\begin{theorem}\label{theorem B}
Let $B$ be a $\Cstar$-algebra and $X\subset B$ a self-adjoint linear subspace. Suppose there exists an $e\in B$ as in \emph{\cref{theorem A}}, and assume $X$ is complete with respect to the norm $\|\cdot\|_{\sbt}$. 
Then the map $\Phi\coloneqq\id_X\colon \Cstar_{\sbt}(X)\longrightarrow  B$ is c.p.c.\ order zero where $\Cstar_{\sbt}(X)$ denotes the associated $\Cstar$-algebra as above. 

When $X$ is closed in $B$, then $e$ is an        order unit and  $\Phi\coloneqq \id_X\colon \Cstar_{\sbt}(X)\\\longrightarrow  B$ is a bounded below c.p.c.\ map with c.p.\ inverse $\Phi^{-1}\colon X\longrightarrow  \Cstar_{\sbt}(X)$ such that $\|\Phi^{-1}\|$ is equal to the scaling factor of $e$. 
When $e$ has scaling factor $1$, $\Phi$ is a complete isometry.
\end{theorem}

This in particular tells us that not only are $X$ and $\Cstar_{\sbt}(X)$ equal as $^*$-vector spaces, they (and their matrix amplifications) also have the same positive cones, i.e., $x\in X$ is positive in $\Cstar_{\sbt}(X)$ if and only if it is positive in $B$, and the same holds for matrix amplifications.

Now combining \cref{theorem B} with the characterization \eqref{dager}  of order zero maps out of unital $\Cstar$-algebras, 
we get a characterization of images of order zero maps as well as a characterization of when a self-adjoint linear subspace of a $\Cstar$-algebra is completely order isomorphic to a unital $\Cstar$-algebra via an order zero map; cf.\ \cref{image of c.p.c. order zero} and \cref{corollary E.5} below.

\begin{corollary}\label{corollary D}
Let $X$ be a self-adjoint linear subspace of a $\Cstar$-algebra $B$ and $e\in B_+^1$. Then the following are equivalent. 
\begin{enumerate}[label=\textnormal{(\roman*)}]
    \item There exists a unital $\Cstar$-algebra $A$ and a c.p.c.\ order zero map $\theta\colon A\longrightarrow B$ with $\theta(A)=X$ and $\theta(1_A)=e$. 
    \item The element $e$ lies in $X\cap X'$ and is a local order unit for $X$ such that $X^2= eX$, and $X$ is complete with respect to $\|\cdot\|_{\sbt}$.
\end{enumerate}
\end{corollary}

\begin{corollary}\label{Cor E}
Let $B$ be a $\Cstar$-algebra and $X\subset B$ a self-adjoint linear subspace. Then $X$ is completely order isomorphic to a unital $\Cstar$-algebra via a c.p.\ order zero map if and only if $X$ is closed in $B$ and there exists an        order unit $e\in B_+^1\cap X\cap X'$ for $X$ with scaling factor $1$ so that $X^2=eX$. 
\end{corollary}

Since the construction of $\Cstar_{\sbt}(X)$ heavily utilizes the local order unit $e$, it is natural to ask how much it depends on the choice of $e$. Below we summarize some of our results in this regard (cf.\ \cref{mult is order unit}, \cref{order unit invariant}, and \cref{envelope}). 

\begin{proposition}
\label{proposition F}
Let $B$ be a $\Cstar$-algebra and $X\subset B$ a closed self-adjoint linear subspace. Suppose there exists a local order unit $e\in B_+^1\cap X'$ for $X$ so that $X^2\subset eX$. 
\begin{enumerate}[label=\textnormal{(\roman*)}]
\item If $e\in X$ and there exists $h\in B_+^1$ which satisfies $hx=ex$ for all $x\in X$, then $h\in X'$ is an        order unit  
 for $X$ with the same scaling factor as $e$, such that $X^2\subset hX$.  Moreover, the associative bilinear maps and norms induced by $e$ and $h$ on $X$ agree, and the induced $\Cstar$-algebras are equal.
\item If there exists another local order unit $h\in B_+^1\cap X\cap X'$ so that $X^2\subset hX$, then the induced $\Cstar$-algebras are unitally $^*$-isomorphic. 
\item If $e\in X$ and has scaling factor $1$, then the abstract operator system $(X,\{\M_r(X)\cap \M_r(B)_+\}_r,e)$ is completely order isomorphic to $\Cstar_{\sbt}(X)$, and so $\Cstar_{\sbt}(X)$ is canonically isomorphic to the enveloping $\Cstar$-algebra for this system. 
\end{enumerate}
\end{proposition}

In several of the above results, the local order unit $e$ for $X$ need not actually lie in $X$. In this situation, one can replace $X$ with $\Span\{X,e\}$ and obtain a unitization of $\Cstar_{\sbt}(X)$ and of the map $\Phi$ from \cref{theorem B}.

\begin{theorem}\label{theorem C}
    Let $B$ be a $\Cstar$-algebra and $X\subset B$ a self-adjoint linear subspace. Suppose there exists an $e\in B$ as in \emph{\cref{theorem A}}, and assume $X$ is complete with respect to the norm $\|\cdot\|_{\sbt}$. Then we have the following:

    \begin{enumerate}[label=\textnormal{(\roman*)}]
\item  The associative bilinear map from \emph{\cref{theorem A}} extends to an associative bilinear map on $Z\coloneqq \Span\{X,e\}$ which also satisfies \eqref{mult id''}, and the norm $\|\cdot\|_{\sbt}$ from \emph{\cref{theorem A}} extends to a $\Cstar$-norm on $Z$. %also given as in \eqref{bullet norm};  
    
\item The inclusion $X\subset Z$ induces an inclusion of $\Cstar_{\sbt}(X)$ in $\Cstar_{\sbt}(Z)$ as a two-sided closed ideal. %, which has co-dimension 1 unless $e\in X$.
When $e\in X$ or $\Cstar_{\sbt}(X)$ is essential in $\Cstar_{\sbt}(Z)$, we may identify the smallest unitization $\Cstar_{\sbt}(X)^\sim$ with $\Cstar_{\sbt}(Z)$; otherwise $\Cstar_{\sbt}(X)$ is already unital, and $\Cstar_{\sbt}(Z)$ is a forced unitization, $\Cstar_{\sbt}(X)^\dagger\cong \Cstar_{\sbt}(Z)$.   
\item The map $\Phi$ from \emph{\cref{theorem B}} extends to a c.p.c.\ order zero map $\Phi^\dagger\coloneqq\id_Z\colon \Cstar_{\sbt}(Z)\longrightarrow B$. In case we identify $\Cstar_{\sbt}(X)^\sim$ with $\Cstar_{\sbt}(Z)$ as above, $\Phi^{\dagger}$ is bounded below  
(resp. a complete order isomorphism) if and only if $\Phi$ is.
\end{enumerate}
\end{theorem}

\medskip

The article is organized as follows. In \cref{sect: preliminaries}, we recall relevant definitions and results regarding completely positive maps, order zero maps, and complete order embeddings. We also give our characterization of order zero complete order embeddings and show by examples that it fails without the order zero hypothesis. In \cref{sec: order unit}, we investigate properties of (commuting) order units which will feature heavily throughout later sections. \cref{Section:Def bullet} is dedicated to establishing \cref{theorem A}, and \cref{sect: a c.p.c. order zero map} to the c.p.c.\ order zero map $\Phi$ from \cref{theorem B}. In \cref{Z}, we consider unitizations of our construction, establishing most of \cref{theorem C}. We also address here the dependence of our constructions on the choice of order unit. Finally, in \cref{closure} we incorporate the further assumption that our subspace is closed and strengthen results from the previous sections accordingly, completing our proof of \cref{theorem B}.

\renewcommand*{\thetheorem}{\roman{theorem}}
\numberwithin{theorem}{section}
\renewcommand*{\thedefinition}{\roman{definition}}
\numberwithin{definition}{section}
\renewcommand*{\theproposition}{\roman{theorem}}
\numberwithin{proposition}{section}
\renewcommand*{\thecorollary}{\roman{corollary}}
\numberwithin{corollary}{section}
%===============================================================
\section{Order zero order embeddings}\label{sect: preliminaries}
\noindent For a $\Cstar$-algebra $A$, we 
write $A_+$ for the cone of positive elements, $A_{\mathrm{s.a.}}$ for its set of self-adjoint elements, and $A^1$, $A^1_+$, and $A^1_{\mathrm{s.a.}}$ for the respective closed unit balls. 
We denote the minimal unitization of $A$ by $A^{\sim}$ with the convention that $A=A^{\sim}$ when $A$ is unital, and we write $A^\dagger$ for the smallest forced unitization of $A$, which agrees with $A^{\sim}$ when $A$ is non-unital and contains $1_A$ as a central co-rank 1 projection when $A$ is unital. The multiplier algebra of $A$ (regarded as a subalgebra of the double dual $A^{**}$) is denoted by $\mathcal{M}(A)$. 
 For a self-adjoint linear subspace $X$ of a larger $\Cstar$-algebra $M$ containing $A$ and an element $e\in A$, we write $X' \cap A\subset M$ for the relative commutant of $X$ in $M$, and we define
 \begin{align*}
   X^2&\coloneqq\{xy\mid x,y\in X\}\ \text{ and}\\
   eX&\coloneqq\{ex\mid x\in X\}.
 \end{align*}
In the literature $X^2$ sometimes denotes $\Span\{xy\mid x,y\in X\}$; throughout this article, the two will coincide because of assumptions placed on $X^2$. For $r\geq 1$, we write $\M_r(X)$ for $r \times r$ matrices over $X$. 
We will fix throughout an approximate inverse $\{h_j\}_j\in C_0((0,1])_+\subset C([0,1])_+$ for $\id_{[0,1]}$ by defining for each $j\geq 1$
 \begin{align}\label{hk}
   h_j(t) = \left\{
     \begin{array}{lr}
       j^2t & : t \in [0,\frac{1}{j}]\\
       \frac{1}{t} & : t \in [\frac{1}{j},1]
     \end{array}
   \right..
\end{align}

%-------------------------------------------------------------------------------

Next we highlight a few relevant properties and observations for completely positive maps and completely positive order zero maps.  For a full introduction to completely positive and completely bounded maps and see \cite{Pau02}.
\begin{definition}
Let $A$ and $B$ be $\Cstar$-algebras with self-adjoint linear subspaces $X\subset A$, $Y\subset B$, and $\theta\colon X\longrightarrow Y$ a linear map. We say $\theta$ is \emph{positive} if $\theta(x)\in Y\cap B_+$ for all $x\in X\cap A_+$. We say it is \emph{completely positive} (c.p.) if this holds for all matrix amplifications $\theta^{(r)}\colon \M_r(X)\longrightarrow \M_r(Y)$.  
We say $\theta$ is \emph{completely contractive} if $\sup_{r\geq 1} \|\theta^{(r)}\|\leq 1$. We say a map is c.p.c.\ when it is completely positive and (completely) contractive. 

We say $\theta$ is a \emph{complete order embedding} if $\theta$ is c.p.\ and completely isometric  
with c.p.\ inverse $\theta^{-1}\colon \theta(A)\longrightarrow A$. A surjective complete order embedding is a \emph{complete order isomorphism}.  
We say $\theta$ is \emph{bounded below} if there exists $R>0$ so that $\|\theta(x)\|\geq R\|x\|$ for all $x\in X$. 
\end{definition}

\begin{remarks}\label{poster}

(i) Under certain assumptions, some of the properties of a complete order embedding are implied by the others, which is why this term tends to have varying definitions in the literature. Many of these require that the map is unital, but even in the nonunital setting, 
%It follows from (i) that a u.c.p.\ map with (u)c.p.\ inverse is a complete order embedding. Without positivity, the Haagerup-Paulsen-Wittstock dilation theorem (or \cite[Corollary 5.12]{ER00}) implies that a \emph{unital} completely isometric embedding is always a complete order embedding (see \cite[Corollary B.9]{BO}). On the other hand, 
if the domain of an isometric c.p.\ map is a (not-necessarily unital) $\Cstar$-algebra, then the map is automatically a complete order embedding if it is completely isometric or order zero 
(see \cite[Remark 1.7(ii)]{CW1} and
\cite[Proposition 1.8]{CW1}, 
%and \cite[Theorem II.6.9.17]{Bla06}, 
respectively).

(ii) Complete order isomorphisms between $\Cstar$-algebras are automatically $^*$-homomorphisms (see e.g., \cite[Theorem II.6.9.17]{Bla06}). It follows that a self-adjoint linear subspace $X$ of a $\Cstar$-algebra $B$ which is completely order isomorphic to a $\Cstar$-algebra $A$ captures the isomorphism class of $A$ in that another $\Cstar$-algebra $C$ is $^*$-isomorphic to $A$ if and only if it is completely order isomorphic to $X$. 

(iii) 
Bounded below maps are exactly the invertible bounded linear maps with bounded linear inverses (defined on the image). 
A linear map between Banach spaces is bounded below if and only if it is injective with closed range.
It follows that all matrix amplifications of a bounded below map $\theta\colon X\longrightarrow Y$ between closed self-adjoint linear subspaces of $\Cstar$-algebras are bounded below.
However, the lower bound need not be uniform, i.e., $\theta^{(r)}$ has a bounded linear inverse for each $r\geq 1$, but we are not guaranteed that $\theta^{-1}$ is completely bounded. Even 
an isometric c.p.\ map between closed self-adjoint linear subspaces of $\Cstar$-algebras with c.p.\ inverse 
can fail to be completely isometric as the following example shows. 
\end{remarks}

\begin{example}\label{Bestellung}
We construct a map $\varphi\colon \M_3\longrightarrow A$, where $A$ is a homogeneous $\Cstar$-algebra, such that $\varphi$ is c.p.\ and isometric with c.p.\ inverse (defined on its image) but is \emph{not} completely isometric. 

We denote by $\mathrm{Gr}(2,\C^3)$ the Grassmannian manifold of all $2$- dimensional subspaces of $\C^3$, which is a compact Hausdorff space. 
Set \[B\coloneqq C(\mathrm{Gr}(2,\C^3),\M_3),\] and define $\theta\colon \M_3\longrightarrow B$ by $\theta(a)(V)\coloneqq p_Vap_V$ where $p_V\in \M_3$ is the projection onto $V$. 
Define 
\begin{align*}
    \varphi\coloneqq\theta\oplus \textstyle{\frac{1}{2}}\id_{\M_3}\colon \M_3 &\longrightarrow B\oplus \M_3=:A.
\end{align*} 

First, we claim that $\varphi$ is c.p.\ and isometric, which comes down to checking that $\theta$ is c.p.\ and isometric. Since \[\theta^{(r)}\colon \M_r(\M_3)\longrightarrow \M_r\big(C(\mathrm{Gr}(2,\C^3)),\M_3\big)\cong C\big(\mathrm{Gr}(2,\C^3),\M_r(\M_3)\big),\] to show that $\theta$ is c.p.,  it suffices to show that $\text{ev}_V\circ \theta$ is c.p.\ for every $V\in \mathrm{Gr}(2,\C^3)$. But $\text{ev}_V\circ \theta$ is simply the compression by a projection, which gives a c.p.\ map. 
To see that 
$\theta$ is isometric, set $a\in \M_3\backslash\{0\}$ and $0\neq \xi\in \C^3$ such that $\|a\xi\|=\|a\|\|\xi\|$, and choose $V\in \mathrm{Gr}(2,\C^3)$ such that $\xi$ and $a\xi\in V$. Then $\|\theta(a)(V)\|=\|p_Vap_V\|=\|a\|$, and hence $\|\theta(a)\|\geq \|a\|\geq \|\theta(a)\|$. 

We claim that $\varphi$ is not 2-isometric. Since the second summand, $\frac{1}{2}\id_{\M_3}$ is strictly contractive, we only need to demonstrate that $\theta^{(2)}$ is not isometric. Consider 
\begin{align*}
    s =%\textstyle{\frac{1}{\sqrt{2}}} \begin{pmatrix} 0 & e_{13}\\ 0 & e_{23} \end{pmatrix}=
    \textstyle{\frac{1}{\sqrt{2}}} (e_{12}\otimes e_{13} + e_{22}\otimes e_{23})\ %\in \M_2(\M_3).
    \in \M_2\otimes \M_3.
\end{align*}
Then $\|s\|=1$, $s$ is rank 1, and 
\begin{align}
%s^*s&=\begin{pmatrix} 0 & 0 \\ 0 & e_{11}\end{pmatrix}\ \text{ and }\\
%ss^*&=\textstyle{\frac{1}{2}}\begin{pmatrix} e_{11} & e_{12}\\ e_{21} & e_{22}\end{pmatrix}. 
    s^*s&=e_{22}\otimes e_{33}\ \text{ and } \label{Bes1}\\
    ss^*&=\textstyle{\frac{1}{2}} (e_{11}\otimes e_{11} + e_{12} \otimes e_{12} + e_{21}\otimes e_{21} + e_{22}\otimes e_{22}).\label{Bes2}
\end{align}
Suppose there exists a $V\in \mathrm{Gr}(2,\C^3)$ so that
\[\|\theta^{(2)}(s)(V)\|= 
\|(1_{\M_2}\otimes p_V) s (1_{\M_2}\otimes p_V)\|=\|s\|=1.\]
Then there exists $\xi\in V\oplus V$ with $\|\xi\|=1$ such that $\|(1_{\M_2}\otimes p_V) s\xi\|=\|s\xi\|=\|s\|\|\xi\|=1$, and hence $s\xi\in V\oplus V$. Since $s$ is rank 1, that means $V\oplus V$ is an invariant subspace for $s$, and hence $(1_{\M_2}\otimes p_V) s =s (1_{\M_2}\otimes p_V)$. It follows that $ss^*\geq (1_{\M_2}\otimes p_V)ss^*=ss^*(1_{\M_2}\otimes p_V)\neq 0$ and likewise with $s^*s$. 
%Then $s^*s\leq (1_{\M_2}\otimes p_V)$ and $ss^*\leq (1_{\M_2}\otimes p_V)$ (since these are rank 1 projections). %, and so $e_{33}, e_{11}, e_{22}\leq p_V$. 
Moreover, from $[ss^s,1_{\M_2}\otimes p_V]=[s^*s,1_{\M_2,p_V}]=0$, \eqref{Bes1}, and \eqref{Bes2} it follows that $[p_V,e_{11}]=[p_V,e_{22}]=[p_V,e_{33}]=0$. Since $(1_{\M_2}\otimes p_V)s^*s=e_{22}\otimes p_Ve_{33}$ is positive and nonzero, we see immediately that $p_Ve_{33}\neq 0$. Combined with the fact that $[p_V,e_{33}]=0$, we have $p_V\geq e_{33}$. Since $(1_{\M_2}\otimes p_V)ss^*=\textstyle{\frac{1}{2}} (e_{11}\otimes p_Ve_{11} + e_{12} \otimes p_Ve_{12} + e_{21}\otimes p_Ve_{21} + e_{22}\otimes p_Ve_{22})$ is positive and nonzero, we have  
%$|\langle p_Ve_{12}\xi,\eta\rangle|^2 \leq \langle pe_{11}\eta,\eta\rangle \cdot \langle pe_{22}\xi,\xi\rangle$ for all $\xi,\eta\in \C^3$. Hence 
$p_Ve_{11}\neq 0$ or $p_Ve_{22}\neq 0$. If $p_Ve_{22}=0$, then $(1_{M_2}\otimes p_V)ss^*=\frac{1}{2}(e_{11}\otimes e_{11})$, contradicting $ss^*\geq (1_{M_2}\otimes p_V)ss^*$. %, it follows that $e_{12} \otimes e_{12} + e_{21}\otimes e_{21} + e_{22}\otimes e_{22}\geq 0$, %since $e_{11}$ is rank 1 and $p_V{e_11}\leq e_{11}$ is a projection
%a contradiction. 
Likewise we must also have $p_Ve_{11}\neq 0$, and so $p_V\geq e_{11},e_{22}$ as well. %\cite[Exercise3.2]{Pau02}
But this means $\dim(V)\geq 3$, a contradiction. 

Hence we must have $\|\theta^{(2)}(s)(V)\|<1$ for all $V\in \mathrm{Gr}(2,\C^3)$. Since $\mathrm{Gr}(2,\C^3)$ is compact, that means $\|\theta^{(2)}(s)\|<1=\|s\|$, and therefore $\theta$ is not 2-isometric, let alone completely isometric.
\end{example}

This demonstrates that, in the non-unital setting, even if the domain is a $\Cstar$-algebra, an isometric and c.p.\ map with c.p.\ inverse may fail to be a complete order embedding, whence assumptions are necessary.

\begin{definition}
A c.p.\ map $\theta\colon A\longrightarrow B$ between $\Cstar$-algebras is \emph{order zero} 
if for any $a,b\in A_+$, 
\[ab=0\ \Longrightarrow\ \theta(a)\theta(b)=0.\]
\end{definition}
While by Stinespring's Theorem, every c.p.\ map can be written as a compression of a $^*$-homomorphism, order zero maps are characterized in \cite[Theorem 3.3]{WZ09} as compressions of $^*$-homomorphisms by commuting positive elements. 
\begin{theorem}\label{thm: structure thm} \emph{(Structure theorem for order zero maps)}
Let $A$ and $B$ be $\Cstar$-algebras, $\theta\colon A\longrightarrow B$ a c.p.\ order zero map, $C=\Cstar(\theta(A))$, $\{u_\lambda\}_\lambda$ an increasing approximate identity for $A$, and $h=\emph{s.o.-}\lim_\lambda \theta(u_\lambda)\in C^{**}$. Then  $0\leq h\in \mathcal{M}(C)\cap C'\subset C^{**}$ 
with $\|h\|=\|\theta\|$, and there exists a $^*$-homomorphism $\pi_\theta\colon A\longrightarrow \mathcal{M}(C)$ so that for all $a\in A$, 
\[\theta(a)=h\pi_\theta(a)=h^{1/2}\pi_\theta(a)h^{1/2}.\]
\end{theorem}
\begin{remarks}\label{cor: order zero for unital}
(i) (\cite[Proposition 3.2]{WZ09}) When $A$ is unital, we have $h=\theta(1_A)$, and in general the map $\theta^{\sim}\colon A^{\sim}\longrightarrow \mathcal{M}(C)$ given by $\theta^{\sim}(a+\lambda 1_{A^\sim})=\theta(a)+\lambda h$ is the unique c.p.\ order zero extension of $\theta$. 
In any case, $h\geq \theta(a)$ for all $a\in A_+^1$. Note that for $\theta^{\sim}$ to be order zero, one can replace $h$ with $1_{\mathcal{M}(C)}$ only if $\theta$ is in fact multiplicative, since u.c.p.\ order zero maps are automatically $^*$-homomorphisms. 

(ii) As pointed out in the introduction, the structure theorem gives a useful characterization of when a c.p.\ map with unital domain is order zero: 
    Assume $A$ and $B$ are $\Cstar$-algebras with $A$ unital. Then a c.p.\ map $\theta\colon A\longrightarrow B$ is order zero if and only if \eqref{dager} holds, i.e., $\theta(a)\theta(b)=\theta(1_A)\theta(ab)$, for all $a,b\in A$. 

(iii) Using the orthogonal decomposition of self-adjoint elements, one can show that for $\theta \colon A\longrightarrow B$ c.p.\ order zero, for any $a\in A$ with $\theta(a)\geq 0$ there exists some $b\in A_+$ with $\theta(a)=\theta(b)$. It follows that for all $r\geq 1$
 \[\theta^{(r)}(\M_r(A)_+)=\M_r(\theta(A))\cap \M_r(B)_+.\]
When $\theta$ is moreover assumed to be injective, for any $r\geq 1$ and $a\in M_r(A)$ we have $\theta^{(r)}(a)\geq 0$ if and only if $a\geq 0$. 
\end{remarks}

\begin{proposition}\label{prop: isom perp is coe}\label{a: isometric unitization}
Suppose $\theta\colon A\longrightarrow B$ is a c.p.\ order zero map. Then the following are equivalent. 

\begin{enumerate}[label=\textnormal{(\roman*)}]
    \item $\theta$ is isometric on positive elements of $A$.

\item $\theta$ is isometric.

\item $\theta$ is a complete order embedding.

\end{enumerate}
Moreover, if $A$ is nonzero and if $\theta$ is a complete order embedding which extends to a c.p.c.\ order zero map $\theta^\dagger\colon A^\dagger\longrightarrow B$ on the (forced) unitization $A^\dagger$ of $A$, 
then $\theta^\dagger$ is also a complete order embedding. 
\end{proposition}

\begin{proof}
Clearly (iii) implies (i), and that (i) implies (ii) is a direct consequence of \cref{thm: structure thm}, since for any $a\in A$   
\begin{align}
\|\theta(a)\|^2&=\|\theta(a)^*\theta(a)\|\label{eq: nom on pos}\\
&=\|h^2\pi_\theta(a^*a)\|\nonumber\\
&=\|h^2\pi_\theta(|a|^2)\|\nonumber \\
&=\|h^2\pi_\theta(|a|)^2\|\nonumber\\
&=\|\theta(|a|)^2\|\nonumber\\
&=\|\theta(|a|)\|^2.\nonumber 
  \end{align}
That (ii) implies (iii) was shown in \cite[Proposition 1.8]{CW1}. 

 Now suppose $\theta$ is a complete order embedding that extends to a c.p.c.\ order zero map on the (forced) unitization $A^\dagger$ of $A$. 
Note that $\|\theta^\dagger(1_{A^\dagger})\|=1$ since $\theta^\dagger$ was assumed to be c.p.c.\ and $\theta^\dagger(1_{A^\dagger})\geq \theta(a)$ for any $a\in A_+$ with $\|a\|=\|\theta(a)\|=1$.  (We want $A\neq \{0\}$ here.)
By the previous implications, we only need to show that $\theta^\dagger$ is isometric on any fixed $a + \lambda 1_{A^\dagger}\in A^\dagger_+$. Note that in this case $a=a^*$, and we may assume $a\neq 0$ and $\lambda\geq 0$. 

For any $0<\varepsilon<\|a\|$ 
we define $g_\varepsilon\in C_0((0,\|a\|])$ to be linear on $(0,\varepsilon]$ and $1$ on $(\varepsilon, \|a\|]$. Set $u_\varepsilon\coloneqq g_\varepsilon(|a|)$ and define $b_\varepsilon= (a_+-\varepsilon)_+-(a_--\varepsilon)_+$ to be the difference of the $\varepsilon$-cutdowns of $a_+$ and $a_-$, and note that $b_\varepsilon +\lambda 1_{A^\dagger}\geq 0$ by functional calculus. %=g_\varepsilon(a).  
Then we have $u_\varepsilon\in A_+$ with $\|u_\varepsilon\|=1$ and $b_\varepsilon\in A$ with $b_\varepsilon=b_\varepsilon^*$, $\|a-b_\varepsilon\|<\varepsilon$, $0\leq b_\varepsilon+\lambda u_\varepsilon$, and $u_\varepsilon b_\varepsilon=b_\varepsilon u_\varepsilon=b_\varepsilon$. Since $0\in \sigma(b_\varepsilon)$,  for $\chi_{\{1\}}(u_\varepsilon)\in A^{**}$, we can identify $\Cstar(b_\varepsilon)^\sim\cong \Cstar(b_\varepsilon, \chi_{\{1\}}(u_\varepsilon))\cong \Cstar(b_\varepsilon, 1_{A^\dagger})$, 
which gives
\begin{align*}
    \|b_\varepsilon + \lambda 1_{A^{\dagger}}\|%&=\|b_\varepsilon + \lambda \chi_{\{1\}}(u_\varepsilon)\|\\
    &=\|(b_\varepsilon + \lambda \chi_{\{1\}}(u_\varepsilon))\|\\
    &=\|(b_\varepsilon + \lambda u_\varepsilon) \chi_{\{1\}}(u_\varepsilon)\|\\
    &\leq \|b_\varepsilon+\lambda u_\varepsilon\|.
\end{align*} %Since $\|u_\varepsilon\|=1$, fctnl calc gives $u_\varepsilon\chi_{\{1\}}(u_\varepsilon)=\chi_{\{1\}}(u_\varepsilon)$. 
Since $\theta$ is isometric on $A$, we have 
\begin{align*}
    \|a + \lambda 1_{A^{\dagger}}\| &\leq \varepsilon + \|b_\varepsilon + \lambda 1_{A^{\dagger}}\| \\
    &\leq \varepsilon + \|b_\varepsilon+\lambda u_\varepsilon\|\\ 
    &=  \varepsilon + \|\theta(b_\varepsilon+\lambda u_\varepsilon)\|\\
    &\leq \varepsilon +  \|\theta^{\dagger}(b_\varepsilon + \lambda 1_{A^{\dagger}})\|\\
    &\leq 2 \varepsilon + \|\theta^{\dagger}(a + \lambda 1_{A^{\dagger}})\|.
\end{align*}
It follows that $\theta^{\dagger}$ is isometric on $A^{\dagger}_+$ and hence also on $A^{\dagger}$.
\end{proof}

%================================================================================================================================

\section{Order Units}\label{sec: order unit}

\noindent Drawing inspiration from order units for abstract operator systems, 
we study local order units for a
(not-necessarily-unital) 
self-adjoint linear subspace $X$ of a $\Cstar$-algebra $B$. Using the ambient $\Cstar$-algebraic structure, we can place more sophisticated assumptions on our local order unit, in particular that $e\in X'$, and then take full advantage of the functional calculus in $B$ to draw conclusions about $X$ and the pair $(X,e)$. 

\begin{definition}\label{span(X,e)}
Let $B$ be a $\Cstar$-algebra and $X\subset B$ a self-adjoint linear subspace. We say $e\in B^1_+$ is a \emph{local order unit} for $X$ if for any self-adjoint element $x\in X$, there exists an $R>0$ so that $Re \geq x$. We call $e$ a \emph{commuting local order unit} for $X$ if $e$ is a local order unit for $X$ and $e\in X'$. 
\end{definition}

\begin{remarks}\label{automatically mou}
(i) We can view a self-adjoint linear subspace $X\subset B$ of a $\Cstar$-algebra as a matrix ordered $^*$-vector space as in \cite[Chapter 13]{Pau02} 
by endowing it with the matrix ordering $\{\M_r(X)\cap \M_r(B)_+\}_r$  it inherits from $B$. In this setting a local order unit for a self-adjoint linear subspace of a $\Cstar$-algebra $B$ 
is a local order unit in the sense of Choi and Effros (see \cite[Chapter 13]{Pau02}) 
for the matrix ordered $^*$-vector space $(X,\{\M_r(X)\cap \M_r(B)_+\}_r)$. 
(Note that if $e$ is a local order unit for $X\subset B$, then it is automatically a matrix local order unit, i.e., $e^{(r)}$ is a local order unit for $\M_r(X)$ for all $r\geq 1$. It is also Archimedean in the sense of \cite[Chapter 13]{Pau02} since $\M_r(B)_+$ is closed.) 

(ii) A novelty of our definition is that we do not require that $e\in X$, in which case it is possible that $X\cap B_+= \{0\}$ with $X\neq \{0\}$, and we are only guaranteed that $X$ contains nontrivial self-adjoint elements. (This problem will disappear when we incorporate the further assumption that $X^2\subseteq eX$, see \cref{X has positive elements}.) 
\end{remarks}

\begin{examples}\label{Order units in a unital A}

(i) 
Let $X$ be a self-adjoint linear subspace of a unital $\Cstar$-algebra $B$ with $1_B\in X$. Then any local order unit of $X$ must be invertible in $B$ by virtue of dominating $1_B$ up to some scalar factor, and so the collection of local order units for $X$ is exactly $\text{GL}(B)\cap B^1_+$. 

(ii) The subspace $C_c((0,1])$ of the non-unital $\Cstar$-algebra $C_0((0,1])$ has local order unit
    $\id_{(0,1]}$. 

(iii) 
    If $\theta\colon A\longrightarrow B$ is a c.p.c.\ map from a unital $\Cstar$-algebra, then $\theta(1_A)$ is a local order unit for $\theta(A)$. 

(iv) It follows from \cref{thm: structure thm} that if $\theta\colon A\longrightarrow B$ is a c.p.c.\ order zero map from a unital $\Cstar$-algebra, then $\theta(1_A)$ is a  commuting local order unit for the self-adjoint linear subspace $\theta(A)\subset B$. 
Even if $A$ is non-unital, we can, in some cases (see for example \cref{unitize order zero'}), leverage the structure theorem for c.p.\ order zero maps (as in Remark \ref{cor: order zero for unital}(i)) to find a c.p.c.\ order zero extension $\theta^{\sim}\colon A^{\sim}\longrightarrow B$ with $\theta^{\sim}(1_{A^{\sim}})=h$ (with $h$ as in Remark \ref{cor: order zero for unital}(i)). In this case $\theta^{\sim}(1_{A^{\sim}})\in B^1_+\cap \theta(A)'$ is a commuting local order unit that is typically \emph{not} in $\theta(A)$. 
\end{examples}

The following two propositions give  technical facts for commuting local order units which will be used throughout the paper.

\begin{proposition}\label{commuting order unit}
    Let $B$ be a $\Cstar$-algebra. 
     \begin{enumerate}[label=\textnormal{(\roman*)}]
     \item For $y,z\in B$ such that $z\geq 0$ and $yz=zy$,\label{eq: hashtag}
    \begin{align}
    z\geq |y|\ \Longleftrightarrow\ z|y|\geq y^*y.\label{comm order unit i} 
    \end{align}
    \item   If $X\subset B$ is a self-adjoint linear subspace, then the following are equivalent for 
      $e\in B^1_+\cap X'$: \label{commuting order unit (b)}
          \begin{enumerate}[label=\textnormal{(\alph*)}]
          \item $e$ is a local order unit for $X$.
          \item For each $x\in X$, there exists an $R>0$ so that $Re\geq |x|$. 
      \end{enumerate} 
     \item  When either of the equivalent conditions in \ref{commuting order unit (b)} hold, the $\Cstar$-algebra $\Cstar(X)$ generated by $X$ in $B$ is contained in the hereditary $\Cstar$-algebra $\overline{eBe}$ generated by $e$. In particular, for any $a\in \Cstar(X)$, $\lim_j \|h_j(e)ea - a\| =0$, where $\{h_j\}_j$ are as in \eqref{hk}. \label{approx id}
       \end{enumerate}
\end{proposition}

\begin{proof}
(i) For $s,t\in [0,\infty)$, we have $s\geq t \longleftrightarrow st\geq t^2$. Hence \eqref{comm order unit i} follows from the fact that there is an isomorphism $\Cstar(z,y^*y)\cong C_0(\Omega)$ for some locally compact $\Omega\subset \sigma(z)\times \sigma(y^*y)$  which sends $z$ and $|y|$ to the functions $((s,t)\mapsto s)$ and $((s,t)\mapsto t)$, respectively.

(ii)
(b) $\Longrightarrow$ (a) since $|x|\geq x$ for all self-adjoint $x\in X$. To see that (a) $\Longrightarrow$ (b), fix $x\in X$. Then there exists $R>0$ so that $Re\geq x^*+x$ and $Re\geq i(x-x^*)$. Since $e$ and $x+x^*$ commute, we also obtain $R^2e^2\geq (x+x^*)^2$ and likewise $R^2e^2\geq (i(x-x^*))^2$. Then 
\begin{align*}
    R^2e^2&\geq \textstyle{\frac{1}{2}}((x+x^*)^2-(x-x^*)^2)=x^*x+xx^*\geq x^*x,
    %&=\frac{1}{2}(x^2+x^*x+xx^*+x^{*2}-x^2+x^*x+xx^*-x^{*2})\\
    %&=x^*x+xx^*\geq x^*x.
\end{align*} 
and hence $Re\geq |x|$.

(iii) To see that $\Cstar(X)\subset \overline{eBe}$, it suffices to show that $x \in \overline{eBe}$ for any self-adjoint $x\in X$. For $x=x^*\in X$ there exists $R>0$ so that $Re\pm x\geq 0$, and so 
$Re\geq |x|\geq 0$. 
Hence $\Re+|x|\in \overline{eBe}$ and so also  $Re+x\in \overline{eBe}$ since $Re+|x|\geq Re+x\geq 0$. It follows that %$Re+x\in \overline{eBe}$ and 
$x=Re + x -Re\in \overline{eBe}$, and so $\Cstar(X)\subset \overline{eBe}$.
Since $\{h_j(e)e\}_j$ forms an (increasing) approximate identity of positive contractions for $\overline{eBe}$, the final claim follows.
\end{proof}

\begin{proposition}\label{order unit}
Let $B$ be a $\Cstar$-algebra and $X\subset B$ a self-adjoint linear subspace with local order unit $e\in B_+^1\cap X'$. 
Then the following hold for all $x \in \Span\{X,e\}$.
\begin{enumerate}[label=\textnormal{(\roman*)}]
    \item If $e^nx=0$ for some $n\geq 1$, then $x=0$. (Nondegeneracy)
    %\item If $e^2x^*x=0$, then $x=0$.
    \item If $ex=(ex)^*$, then $x=x^*$.
    \item If $ex\geq 0$, then $x\geq 0$. 
    %\item $e$ is an Archimedean order unit for $X$.
    %\item If $xy\in B$ is self-adjoint, then there exists $R>0$ such that $Re^2\geq xy$; in particular, $Re\geq xy$. Moreover $e^2xy=0$ implies $xy=0$. 
\end{enumerate}
Moreover, these claims hold for all matrix amplifications $\M_r(X)$ with order units $e^{(r)}$.
\end{proposition}

\begin{proof}
Note that if $e$ is a  commuting local order unit for $X$, then it will also be a  commuting local order unit for $\Span\{X,e\}$. 
Hence, for notational simplicity, 
we reduce to the assumption $X=\Span\{X,e\}$. 
Fix $x\in X$. 

For (i), we first handle the case where $x=x^*$. Since $ex=xe$, we may identify $\Cstar(e,x)$ with $C_0(\Omega)$ for some locally compact Hausdorff space $\Omega\subset \sigma(e)\times \sigma(x)$. If $e^nx=0$ for some $n\geq 1$, then %$s^nt=0$ for all $(s,t)\in \Omega$. It follows that 
for each pair $(s,t)\in \Omega$ either $s=0$ or $t=0$. Since $e$ is a local order unit, by \cref{commuting order unit}(ii) there exists an $R>0$ such that $Rs\geq |t|$ for all $(s,t)\in \Omega$. Together these imply that $\sigma(x)=\{0\}$, and so $x=0$. % since $x$ is self-adjoint. 

Now, let $x\in X$ be any element. If $e^nx=0$ for some $n\geq 1$, then 
$e^nx^*=(e^nx)^*=0$, and so 
$e^n (i(x^*-x))=0$ where 
$i(x^*-x)$ is self-adjoint. It follows from the part of the proof above that $i(x^*-x)=0$, which means $x$ is self-adjoint. But in that case, our assumption that $e^nx=0$ implies $x=0$. 

From (i) we immediately get (ii). Indeed, if $x\in X$ and $ex=(ex)^*=ex^*$, then $e(x-x^*)=0$, and so $x=x^*$.

For (iii), assume $ex\geq 0$ for some $x\in X$. Then by (ii), we know $x=x^*$, and so we can again identify $\Cstar(e,x)$ with $C_0(\Omega)$ as in the proof of (i). Suppose $t_0\in \sigma(x)\cap(-\infty,0)$, and let $s_0\in \sigma(e)$ such that $(s_0,t_0)\in \Omega$. 
Since $e$ is a local order unit, there exists $R>0$ such that $Rs_0\geq |t_0|>0$. But this implies $s_0>0$ and hence $s_0t_0<0$, a contradiction. 

Finally, since $e$ is automatically a matrix local order unit for $X$ (as mentioned in \cref{automatically mou}), it follows that for any $r\geq 1$, $e^{(r)}\in \M_r(B)_+^1\cap\M_r(X)'$ is a local order unit for $\M_r(X)$, and so claims (i)-(iii) hold for all matrix amplifications. 
\end{proof}

\begin{remark}\label{bdd below perp gives isoclass} 

Suppose $A,B$, and $C$ are $\Cstar$-algebras with $A$ and $B$ unital. If there exist bounded below c.p.\ order zero maps $\psi\colon A\longrightarrow B$ and $\varphi \colon C\longrightarrow B$ with $\psi(A)=\varphi(C)$ and $\psi(1_A)=\varphi(1_C)$, then $\theta\coloneqq \varphi^{-1}\circ \psi\colon A\longrightarrow C$ is a unital $^*$-isomorphism -- regardless of the norms of the maps. Indeed, first note that $\theta$ is a well-defined, bijective, bounded $^*$-linear map, so it suffices to show that $\theta(ab)=\theta(a)\theta(b)$ for fixed $a,b\in A$.  The structure theorem for order zero maps tells us that 
\begin{align}\label{eq: hom}
\psi(1_A)\psi(ab)=\psi(a)\psi(b)=\varphi(\theta(a))\varphi(\theta(b))=\varphi(1_C)\varphi(\theta(a)\theta(b)).
\end{align}
In Example \ref{Order units in a unital A}(iv), we saw that $X\coloneqq \psi(A)=\varphi(C)$ and $e\coloneqq \psi(1_A)=\varphi(1_C)$ satisfy the assumptions of \cref{order unit}, and so applying nondegeneracy (\cref{order unit}(i)) to \eqref{eq: hom} tells us that $\psi(ab)\\ =\varphi(\theta(a)\theta(b))$. So we have 
\begin{align*}
    \theta(ab)=\varphi^{-1}(\psi(ab))=\varphi^{-1}(\varphi(\theta(a)\theta(b)))=\theta(a)\theta(b).
\end{align*}
This fact can also be deduced readily from other results in the literature (e.g. Corollary D with Theorem 2.2 from \cite{GT1} (via \cite{GT2})).
\end{remark}

Notice that for the unit of a $\Cstar$-algebra, the $R>0$ by which we scale $1_A$ in order to dominate some $a\in A_{\mathrm{s.a.}}$ can always be taken to be $\|a\|$. This uniformity motivates the following.

\begin{definition}\label{def 68}
Let $B$ be a $\Cstar$-algebra and $X\subset B$ a self-adjoint linear subspace. We say an element $e\in B_+^1$ is an \emph{order unit}  for $X$ if there exists a $K>0$ so that $K\|x\|e\geq x$ for every $x\in X_{\mathrm{s.a.}}$. 
We call the minimal such $K$ the {\em scaling factor} of $e$. 
\end{definition}
\begin{remarks}\label{rmk: Order scale at least 1} (i) If $X\neq \{0\}$, 
the scaling factor for a commuting         order unit must be at least 1. Indeed, suppose $x\in X_{\text{s.a.}}$ with $\|x\|=1$ and $K\|x\|e=Ke\geq x$ for some $K>0$. Let $\Omega\subset \sigma(e)\times \sigma(x)$ be a locally compact Hausdorff space so that $\Cstar(e,x)\cong C_0(\Omega)$.  Then $K\|x\|e=Ke\geq x$ implies that $Ks\geq |t|$ for all $(s,t)\in \Omega$. Since the spectral radius of $x$ is $1$, we must have some $s\in \sigma(e)\subset [0,1]$ with $Ks\geq 1$. Hence $K\geq 1$. 

(ii) Note also that the 
scaling factor is indeed a minimum since the set $\{K\geq 1 \mid K\|x\|e\geq x,\ \forall\  x\in X_{\mathrm{s.a.}}\}$ is an interval and $B_+$ is closed.

(iii) Just as a bounded map is not automatically completely bounded, if $e$ is an        order unit for $X$, the same scaling factor may not suffice for all matrix amplifications (unless the scaling factor is $1$ and $X$ is closed; cf.\ \cref{completely isometric}). 
\end{remarks}

\begin{examples}\label{ex: bdd below UOU}
(i) We saw in Example \ref{Order units in a unital A}(i) that the local order units for a \emph{unital} self-adjoint linear subspace $X$ of a unital $\Cstar$-algebra $B$ are just $\mathrm{GL}(B)\cap B^1_+$. These are also the        order units, and the scaling factor for $e\in\mathrm{GL}(B)\cap B_+^1$ is the minimum value $R>0$ so that $Re\geq 1_B$, i.e., $\|e^{-1}\|$. 
In particular, that means that if $1_B\in X$, then $1_B$ is the unique 
matrix (local) order unit in $X$ with scaling factor $R=1$.  

(ii) If $\theta\colon A\longrightarrow B$ is a bounded below c.p.\ map 
from a unital $\Cstar$-algebra $A$, then the image of the unit becomes an        order unit for $\theta(A)$ with scaling factor at most $\|\theta^{-1}\|$ since  
\begin{align}\label{eq: order scale upper bd}
    \|\theta^{-1}\|\|\theta(a)\|\theta(1_A)\geq \|a\|\theta(1_A)\geq \theta(a)
\end{align}
for all $\theta(a)\in \theta(A)$. 
\end{examples}

When $\theta$ is order zero, $\|\theta^{-1}\|$ is the scaling factor of $\theta(1_A)$ as the following proposition shows. 

\begin{proposition}\label{invertible c.p. order zero and order scale}
Let $\theta\colon A\longrightarrow B$ be an injective c.p.c.\ order zero map. 
\begin{enumerate}[label=\textnormal{(\roman*)}]
    \item When $A$ is unital, we define $R_0\coloneqq 0$ and \[R_b\coloneqq\min\{R\geq 1 \mid R\|\theta(b)\|\theta(1_A)\geq \theta(b)\}\] for $b\in A_{\mathrm{s.a.}}\backslash\{0\}$. Then $R_{|a|}\|\theta(a)\|=\|a\|$ for each $a\in A\backslash\{0\}$. 
If $\|\theta(a)\|\theta(1_A)\geq \theta(a)$ for all $a\in A_+$, then $\theta$ is a complete order embedding.
\item If $\theta$ is bounded below and extends to a c.p.c.\ order zero map $\theta^{\sim}\colon A^{\sim}\longrightarrow B$, 
then $\theta^{\sim}$ is also bounded below, and  $\theta^{\sim}(1_{A^{\sim}})$ is an        order unit for $\theta(A)$ with scaling factor $\|\theta^{-1}\|$. When $\theta$ is isometric, this is also its scaling factor for $\theta^{\sim}(A^{\sim})$. In particular, the scaling factor of $\theta^{\sim}(1_{A^{\sim}})$ is $1$ if and only if $\theta^\sim$ is completely isometric.
\end{enumerate}
\end{proposition}

\begin{remark}
When $A$ is unital, then $\theta^\sim=\theta$, and \cref{invertible c.p. order zero and order scale}(ii) tells us $\theta(1_A)$ is an order unit for $\theta(A)$ with order scale $\|\theta^{-1}\|$. When $A$ is not unital, applying this same argument to $A^\sim$ tells us $\theta^{\sim}(1_{A^{\sim}})$ is an order unit for $\theta^\sim(A^\sim)$ with scaling factor $\|(\theta^{\sim})^{-1}\|$. 
Also note that the first claim in (ii) is about the scaling factor of $\theta^{\sim}(1_{A^{\sim}})$ for $\theta(A)$, not the scaling factor of $\theta^{\sim}(1_{A^{\sim}})$ for $\theta^\sim(A^\sim)$, which, as we have just explained, is $\|(\theta^{\sim})^{-1}\|$. 
\end{remark}

\begin{proof}
(i) Assume $A$ is unital, and let $b\in A_{\text{s.a.}}\backslash\{0\}$. 
Since $\theta$ is positive, $\textstyle{\frac{\|b\|}{\|\theta(b)\|}}\|\theta(b)\|\theta(1_A)\geq \theta(b)$, and so the interval $\{R\geq 1 \mid R\|\theta(b)\|\theta(1_A)\geq \theta(b)\}$ is nonempty. Since $B_+$ is closed and $\theta$ is continuous, this interval has a minimum, and so $R_b$ is well-defined and moreover $\textstyle{\frac{\|b\|}{\|\theta(b)\|}}\geq R_b$.

Let $c\in A_+$. Since $\theta$ is injective c.p.c.\ order zero and $R_c\|\theta(c)\|\theta(1_A) - \theta(c)\geq 0$, Remark \ref{cor: order zero for unital}(iii) tells us that $\theta^{-1}\big(R_c\|\theta(c)\|\theta(1_A)-\theta(c)\big) = R_c\|\theta(c)\|1_A- c\geq 0$. Hence we must also have $R_c\|\theta(c)\|\geq \|c\|$, and so $R_c\|\theta(c)\|= \|c\|$. Then it follows from \eqref{eq: nom on pos} that $R_{|a|}\|\theta(a)\|= \|a\|$ for any $a\in A$.

Now suppose $R_b=1$ for all $b\in A_+\backslash\{0\}$. Then it follows %from \eqref{eq: R_a}
that $\theta$ is isometric on $A_+$, and hence \cref{prop: isom perp is coe} tells us $\theta$ is a complete order embedding. 

(ii) Now assume $\theta$ is bounded below and $A$ is non-unital. We first observe that $\theta^{\sim}$ is also bounded below, or equivalently injective with closed image. Indeed, as Banach spaces,\footnote{We use here the well-known Banach space fact that a linear subspace of a Banach space $\mathcal{X}$ is closed if and only if its linear span with any finite-dimensional subspace of $\mathcal{X}$ is closed.\label{B-space closures1}} 
$\theta(A)$ is closed if and only if $\theta^{\sim}(A^\sim)=\Span\{\theta(A),\theta^{\sim}(1_{A^{\sim}})\}$ is, and so we only need to verify that $\theta^\sim$ is injective. 
To that end, let $a\in A$ and $\lambda \in \C$ with $\theta^{\sim}(a+\lambda 1_{A^{\sim}})=\theta(a)+\lambda \theta^{\sim}(1_{A^{\sim}})=0$. Then either $\theta(a)=0$ and $\lambda=0$, whence $a+\lambda 1_{A^{\sim}}=0$ since $\theta$ is injective, or $u\coloneqq -a/\lambda\in A$ maps to $\theta^{\sim}(1_{A^{\sim}})$. 
But then $\theta^{\sim}(1_{A^{\sim}})\theta^{\sim}(b)=\theta^{\sim}(u)\theta^{\sim}(b)= \theta^{\sim}(1_{A^{\sim}})\theta^{\sim}(ub)$ for any $b\in A$. 
By nondegeneracy (\cref{order unit}(i)), it follows that $\theta(b)=\theta(ub)$ and similarly $\theta(b)=\theta(bu)$ for all $b\in A$. % Then $ub=b$ and similarly $bu=b$ for all $b\in A$, and so  
It follows from injectivity that $u$ is the unit of $A$, contradicting our assumption. Hence $\theta^{\sim}$ is also bounded below. 

We then know from Example \ref{ex: bdd below UOU}(ii) that $\theta^{\sim}(1_{A^{\sim}})$ is an        order unit for $\theta^{\sim}(A^{\sim})$ and hence also for $\theta(A)$. It remains to check that its scaling factor for $\theta(A)$ is $\|\theta^{-1}\|$.

Let $R$ denote the scaling factor of $\theta^{\sim}(1_{A^{\sim}})$ for $\theta(A)$. Then $R\leq \|\theta^{-1}\|$ by \eqref{eq: order scale upper bd}. Note that if $b\in A_{\mathrm{s.a.}}$ then $R_{b}\leq R_{|b|}$ 
since
\[R_{|b|}\|\theta(b)\|\theta^{\sim}(1_{A^{\sim}})\overset{\eqref{eq: nom on pos}}{=}R_{|b|}\|\theta(|b|)\|\theta^{\sim}(1_{A^{\sim}})\geq \theta(|b|)\geq \theta(b).\]
Putting all this together, we have using (i)
\begin{align*}
\|\theta^{-1}\|&\geq R\\
&= \textstyle{\sup_{a\in A_{\mathrm{s.a.}}} R_a}\\
&= \textstyle{\sup_{a\in A_+} R_{a}}\\
&=\textstyle{\sup_{a\in A}R_{|a|}}\\\
&= \textstyle{\sup_{a\in A\backslash\{0\}}\frac{\|a\|}{\|\theta(a)\|}}\\
&=\|\theta^{-1}\|.
\end{align*}

It follows that $R=1$ if and only if $\theta$ is isometric if and only if $\theta$ is completely isometric if and only if $\theta^\sim$ is completely isometric by \cref{a: isometric unitization}.  
Applying the result now to $\theta^{\sim}$ tells us the scaling factor of  $\theta^{\sim}(1_{A^{\sim}})$ for $\theta^{\sim}(A^{\sim})$ is also $1$. 
\end{proof}

\begin{example}\label{Psi1}
 Using the same arguments as in \cite[Lemma 3.4]{CW1}, one shows that a system of c.p.\ approximations $A\xlongrightarrow {\psi_n}F_n\xlongrightarrow {\varphi_n}A$ of a separable nuclear $\Cstar$-algebra $A$ with $\psi_n$ and $\varphi_n\circ \psi_n$ all c.p.c.\ and with $\limsup_n\|\varphi_n\|<\infty$ induces a bounded below c.p.c.\ map $\Psi\colon A\longrightarrow \prod F_n/\bigoplus F_n$, which is order zero when the downwards maps $(\psi_n)_n$ are approximately order zero. 
When $A$ is unital, $\Psi(1_A)$ is an         order unit  
 with scaling factor $R\leq \limsup_n \|\varphi_n\|$. If we assume $A\xlongrightarrow {\psi_n}F_n\xlongrightarrow {\varphi_n}A$ is a c.p.c.\ approximation, then we automatically get $R=1$. 
\end{example}

It is perhaps enlightening to see when a local order unit fails to be an        order unit. 

\begin{example}\label{not uou}
 (i) Fix $h\coloneqq((n+1)^{-2})_n\in C_0(\N)\subset C_b(\N)$. Define $\theta\colon C_b(\N)\longrightarrow C_b(\N)$ by $x\longmapsto hx$. Then this map is c.p.c.\ order zero and injective, and  hence $\theta(1)=h\in C_b(\N)_+^1$ is a local order unit for $\theta(C_b(\N))$. 
For each $n\in \N$, 
set $x_n\coloneqq(n+1)\delta_n\in C_b(\N)$. Then $\|\theta(x_n)\|=(n+1)^{-1}$ for each $n\in \N$, and for any $R>0$,
\begin{align*}
    \big(R\|\theta(x_n)\|h-\theta(x_n)\big)(n)%&=R(n+1)^{-1}h-\theta(x_n)&\\
    %&=\big(R(n+1)^{-3}-(n+1)^{-1}\big)_n
    =\textstyle{\frac{R-(n+1)^2}{(n+1)^3}}.
\end{align*} 
Hence there is no $R>0$ so that $R\|\theta(x_n)\|h-\theta(x_n)\geq 0$ for all $n\in \N$, and $h$ is not an        order unit for $\theta(C_b(\N))$. 

Note that $\theta$ does not have closed image (and hence is not bounded below). %Indeed, we have $C_c(\N)\subset \theta(C_b(\N))$ but $C_0(\N)\nsubseteq \theta(C_b(\N))$ (e.g. $((n+1)^{-1})_n\in C_0(\N)\backslash \theta(C_b(\N))$).
\medskip 

\noindent (ii) Similarly, consider the $\Cstar$-algebra $C([0,1])$, the element $h\coloneqq\id_{[0,1]}\in C([0,1])$, and the c.p.\ %c order zero 
map $\theta \colon C([0,1])\longrightarrow C_0((0,1])$ given by $\theta(f)=(hf)|_{(0,1]}$ (which we write as just $hf$). Here again, $h=\theta(1)$ is a  
local order unit but not an        order unit. %Indeed, if it were, then there would exist $R>0$ such that $R\|\theta(f)\|\theta(1)+\theta(f)=R\|hf\|h+f\geq 0$ for every $f\in C([0,1])$. 
For each $j\geq 1$, define $h_j\in C([0,1])$ as in \eqref{hk}.
   Then for each $j\geq 1$, we have $\|\theta(h_j)\|=1$, and so there is no single value $R>0$ so that for every $j\geq 1$ and $t\in (0,1]$,
   $$R\|\theta(h_j)\|h(t)-\theta(h_j)(t)=t(R-h_j(t))\geq 0.$$
Again, notice that $\theta\big(C([0,1])\big)$ is not closed in $C_0((0,1])$.
\end{example}

%---------------------------------------------------------------

\section{A $\Cstar$-structure}\label{Section:Def bullet}

\noindent In this section, we prove our main theorem: 

\begin{theorem}\label{theorem57}
Let $B$ be a $\Cstar$-algebra and $X\subset B$ a self-adjoint linear subspace. Suppose there exists a local order unit $e\in B_+^1\cap X'$ for $X$ so that $X^2\subset eX$. 
Then there is an associative bilinear map ${\sbt}\colon X\times X\longrightarrow X$ satisfying 
\begin{align}\label{mult id}
    xy=e(x{\sbt} y)
\end{align}
for all $x,y\in X$, 
so that $(X,{\sbt})$ is a $^*$-algebra. When $e\in X$, then $e$ is the unit for $(X,{\sbt})$. 

Moreover, the map $\|\cdot\|_{\sbt}\colon X\longrightarrow [0,\infty)$ defined for each $x\in X$ by
\begin{align}\label{bullet norm}
    \|x\|_{\sbt}=\lim_j\|h_j(e)x\| 
\end{align}
with $\{h_j\}_j\subset C_0((0,1])_+$ as given in \eqref{hk}, is a pre-$\Cstar$-norm on $(X,{\sbt})$.
\end{theorem}

We begin by establishing further descriptions of $\|\cdot\|_{\sbt}$. 
(Note that the following Proposition does not require the $X^2\subset eX$ assumption.)

\begin{proposition}\label{norm}
Let $B$ be a $\Cstar$-algebra and $X\subset B$ a self-adjoint linear subspace with local order unit $e\in B_+^1\cap X'$. 
Then the map $\|\cdot\|_{\sbt}\colon X\longrightarrow [0,\infty)$ defined in \eqref{bullet norm} is a norm on $X$. Moreover, for each $x\in X$ we have  
\begin{align}
\|x\|_{\sbt}&=\lim_j\|h_j(e)x\|\label{bullet norm 2}\\
&=\sup_j\|h_j(e)x\| \nonumber\\
&=\lim_j\|h_j(e)^2ex\|\nonumber \\
&=\sup_j\|h_j(e)^2ex\|. \nonumber
\end{align}
\end{proposition}

\begin{proof}
Fix $x\in X$. First we show that the limit in \eqref{bullet norm} exists. Since $(h_j)_j$ is increasing, for all $j\geq 1$ we have (using that $x$ and $e$ commute)
$$\|h_j(e)x\|^2=\|x^*h_j(e)^2x\|\leq \|x^*h_{j+1}(e)^2x\|=\|h_{j+1}(e)x\|^2.$$
 So to show the limits exist, it suffices to show that the sequence is bounded. 
 
 Since $e$ is a local order unit for $X$,  
 we can choose for any self-adjoint $y\in X$ an $R>0$ such that $y\leq Re$. Since $e\in X'$, we have %Then 
$\|h_j(e)y\|\leq \|Rh_j(e)e\|\leq R$ for all $j\geq 1$. Applying this to $(x+x^*)/2$ and $(x-x^*)/(2\mathrm{i})$, we conclude that $(\|h_j(e)x\|)_j$ is bounded and moreover \begin{align}\label{mon and bdd}
\|x\|_{\sbt}=\lim_j\|h_j(e)x\|=\sup_j\|h_j(e)x\|<\infty.
\end{align}

It is straightforward to check that $\|\cdot\|_{\sbt}$ is a seminorm. To see that it is a norm, suppose $x\in X$ and $0=\|x\|_{\sbt}=\sup_j\|h_j(e)x\|$.  Because $0\leq e\leq h_j(e)$ for each $j\geq 1$ and $e$ commutes with $x^*x$, %(and hence $ex^*x=x^*ex\leq x^*h_j(e)x=h_j(e)x^*x$) 
we have 
$$\|e|x|\|^2=\|e^2|x|^2\|\leq \|ex^*x\|\leq \|h_j(e)x^*x\|\leq \|x^*\|\|h_j(e)x\|=0$$ for all $j\geq 1$. So $e|x|=0$, $e$ commutes with $|x|$, and, by \cref{commuting order unit}\ref{commuting order unit (b)}, there exists $R>0$ with $Re\geq |x|$. From this, \cref{commuting order unit}\ref{eq: hashtag} tells us $x^*x\leq Re|x|=0$, and hence $x=0$. 

All that remains is to show the additional claim: for each fixed $x\in X\backslash\{0\}$ we have
\begin{align}\label{extra}
    \|x\|_{\sbt}=\lim_j\|h_j(e)^2ex\|=\sup_j\|h_j(e)^2ex\|.
\end{align}
Since $0\leq h_j(e)^4e^2\leq h_{j+1}(e)^4e^2$ and $0\leq h_j(e)^4e^2\leq h_j(e)^2$ for all $j\geq 1$, we have 
\begin{align*}
    \|h_j(e)^2ex\|^2&=\|x^*h_j(e)^4e^2x\|\leq \|x^*h_{j+1}(e)^4e^2x\|=\|h_{j+1}(e)^2ex\|^2\ \text{ and}\\
     \|h_j(e)^2ex\|^2&=\|x^*h_j(e)^4e^2x\|\leq \|x^*h_j(e)^2x\|=\|h_j(e)x\|^2\overset{\eqref{mon and bdd}}{\leq}\|x\|_{\sbt}^2
\end{align*}
for all $ j\geq 1$. This implies the sequence $\big(\|h_j(e)^2ex\|\big)_j$ is bounded and nondecreasing, and so the limit in \eqref{extra} exists and furthermore
\begin{align*}
    \lim_j\|h_j(e)^2ex\|=\sup_j\|h_j(e)^2ex\|\leq \sup_j\|h_j(e)x\|%=\lim_j\|h_j(e)x\|
    =\|x\|_{\sbt}.
\end{align*}
It remains to show that this is an equality. Since $e\in B_+\cap X'$, we have $\|h_j(e)|x|\|=\|h_j(e)x\|$ %via \|h_j(e)^2x^*x\|
and likewise $\|h_j(e)^2e|x|\| =\|h_j(e)^2ex\|$ for all $j\geq 1$, and so we instead show that 
\[\sup_j \|h_j(e)|x|\|\leq \sup_j\|h_j(e)^2e|x|\|.\]
We fix $k\geq 1$ 
and aim to find $j\geq 1$ so that $\|h_j(e)^2e|x|\|\geq \|h_k(e)|x|\|$. 
Since $e$ commutes with $|x|$, we identify $\Cstar(e,|x|)=C_0(\Omega)$ as before for some locally compact Hausdorff space $\Omega\subset\sigma(e)\times\sigma(|x|)$. 
First, we claim that there exists $\varepsilon>0$ such that 
\begin{align}\label{eq: cts}
    \|h_k(e)|x|\|=\sup_{(s,t)\in \Omega, s\geq\varepsilon} |h_k(s)t|.
\end{align}
 Since $(s,t)\longmapsto h_k(s)t$ is a continuous function on $[0,1]\times[0,\infty)$, there exists a $\delta>0$ such that $|h_k(s)t|<\|h_k(e)|x|\|/2$ whenever $|(s,t)|<\delta$ (since $x\neq 0$, we can conclude from \cref{order unit}(ii) along with \cref{commuting order unit}(i) that $\|h_k(e)|x|\|>0$). 
 By \cref{commuting order unit}\ref{commuting order unit (b)}, there exists $R>0$ so that %$Re\geq |x|$, i.e., 
 $Rs\geq t$ for all $(s,t)\in \Omega$. For $s<\varepsilon\coloneqq\delta(R^2+1)^{-1/2}$, we have 
 $$|(s,t)|^2=s^2+t^2\leq s^2+R^2s^2<\delta^2,$$
 which establishes the claim.
 
 Now choose $j\geq k$ such that $1/j<\varepsilon$, and consider $(s,t)\in \Omega$ with $s\geq \varepsilon$. Then $s>1/j$ and so $h_j(s)^2st=\textstyle{\frac{t}{s}}$. 
 If $s> 1/k$, then $$h_k(s)t=\textstyle{\frac{t}{s}}=h_j(s)^2st,$$ and if $s\leq 1/k$, then $$h_k(s)t=k^2st\leq \textstyle{\frac{t}{s}}=h_j(s)^2st.$$ Hence for all $(s,t)\in \Omega$ with $s\geq \varepsilon$, we have 
 $$h_j(s)^2st=\textstyle{\frac{t}{s}}\geq h_k(s)t,$$ which gives us
 \begin{align*}
     \|h_j(e)^2e|x|\|\geq \sup_{(s,t)\in \Omega, s\geq\varepsilon}|h_j(s)^2st|\geq \sup_{(s,t)\in \Omega, s\geq\varepsilon} |h_k(s)t|\overset{\eqref{eq: cts}}{=}\|h_k(e)|x|\|.
 \end{align*}
 This establishes \eqref{extra}.
\end{proof}

We are now ready to prove \cref{theorem57}. Since we are establishing an alternative product on a subspace of a $\Cstar$-algebra, we emphasize here that multiplication in $B$ will always be denoted by the usual concatenation (i.e., $xy$). For example for $e (x\sbt y)$, denotes the $B$-product of $e$ and $x\sbt y$. 
\begin{proof}[Proof of \cref{theorem57}]

The fact that $X^2\subset eX$ along with the nondegeneracy of $e$ (i.e., \cref{order unit}(i)) 
guarantee that for each $(x,y)\in X\times X$, there exists a unique element $z$ in $X$ such that $xy=ez$. We denote this element by $x{\sbt} y$,
and conclude that the map ${\sbt}\colon X\times X\longrightarrow X$ given by $(x,y)\longmapsto x{\sbt} y$ is well-defined. 

The arguments for associativity and bilinearity use nondegeneracy of $e$ in the same way, so we just give one argument for associativity. Let $x,y,z\in X$ and $\lambda\in \mathbb{C}$. Then we have
\begin{align*}
    e^2(x{\sbt}(y{\sbt} z))&=ex(y{\sbt} z)\\
    &=xe(y{\sbt} z)\\
    &=x(yz)\\
    &=(xy)z\\
    &=e(x{\sbt} y)z\\
    &=e^2((x{\sbt} y){\sbt} z).
\end{align*}
By nondegeneracy of $e$ (or actually $e^2$ in this case), it follows that $x{\sbt} (y{\sbt} z)=(x{\sbt} y){\sbt} z$. 

To see that $(X,{\sbt})$ is a $^*$-algebra, we compute for $x,y\in X$
$$e(x{\sbt} y)^*=(e(x{\sbt} y))^*=(xy)^*=y^*x^*=e(y^*{\sbt} x^*).$$ Then it follows from nondegeneracy that $(x{\sbt} y)^*=y^*{\sbt} x^*$. 

When $e\in X$, nondegeneracy also gives
$$0=ex-ex=e(e{\sbt} x) - ex = e(e{\sbt} x-x)\Longrightarrow e{\sbt} x=x,$$ and similarly $x{\sbt} e=x$. 
Hence $e$ becomes the unit for $(X,{\sbt})$. 

\medskip

From \cref{norm}, we know that $\|\cdot\|_{\sbt}\colon X\longrightarrow [0,\infty)$ defines a norm on $X$. 
It remains to show that it  
is a pre-$\Cstar$-norm on $(X,{\sbt})$. 
We begin by checking that $\|\cdot\|_{\sbt}$ is a algebra norm with respect to ${\sbt}\colon X\times X\longrightarrow X$ (i.e., it is submultiplicative). Let $x,y\in X$. With \eqref{bullet norm 2} we have that
\begin{align*}
    \|x{\sbt} y\|_{\sbt}&=\lim_j\|h_j(e)^2e(x{\sbt} y)\|\\
    &=\lim_j\|h_j(e)^2xy\|%=\lim_j\|h_j(e)xh_j(e)y\|
    \\
    &\leq (\lim_j \|h_j(e)x\|)(\lim_j \|h_j(e)y\|)\\
    &=\|x\|_{\sbt}\|y\|_{\sbt}. 
\end{align*}
Likewise we check the $\Cstar$-identity for each $x\in X$ again using \eqref{bullet norm 2}
\begin{align*}
    \|x^*{\sbt} x\|_{\sbt}&=\lim_j\|h_j(e)^2e(x^*{\sbt} x)\|\\
    &=\lim_j\|h_j(e)^2x^*x\|\\
    &=\lim_j\|h_j(e)x\|^2\\
    &=\|x\|_{\sbt}^2.
\end{align*}
Hence $\|\cdot\|_{\sbt}$ is a pre-$\Cstar$-norm on $(X,\sbt)$.
\end{proof}

\begin{definition}\label{defn60}
Let $B$ be a $\Cstar$-algebra and $X\subset B$ a self-adjoint linear subspace with local order unit $e\in B_+^1\cap X'$ with $X^2\subset eX$. Let ${\sbt}\colon X\times X\longrightarrow X$ be the product and $\|\cdot\|_{\sbt}\colon X\longrightarrow [0,\infty)$ the norm from \cref{theorem57}. 
We denote the induced $\Cstar$-algebra $\overline{(X,{\sbt})}^{\|\cdot\|_{_{\sbt}}}$ by $\Cstar_{\sbt}(X)$, and we call this the \emph{$\Cstar$-algebra associated to $(X,e)$}. 
\end{definition}

Note that when $X$ is already complete with respect to $\|\cdot\|_{\sbt}$, $\Cstar_{\sbt}(X)$ and $X$ are equal as $^*$-vector spaces. This is often the case in our applications. 

\begin{examples}\label{prop52.5} 
(i) Let $B$ be a $\Cstar$-algebra and $e\in B_+$ with $\|e\|\leq 1$. Then  $X=\Span\{e\}$ is readily seen to satisfy the criteria of \cref{theorem57}, and in this case $\Cstar_{\sbt}(X)\cong\C$.

(ii) Let $A$ and $B$ be $\Cstar$-algebras and $\theta\colon A\longrightarrow B$ a c.p.c.\ map. Then $X=\theta(A)$ is a self-adjoint linear subspace of $B$, and we can define an associative bilinear map ${\sbt}\colon X\times X\longrightarrow X$ given by $\theta(a){\sbt}\theta(b)=\theta(ab)$. When $A$ is unital and $\theta$ is moreover order zero, $e\coloneqq\theta(1_A)\in B_+^1\cap X\cap X'$ is a local order unit for $X$ and $eX=X^2$ by the structure theorem for order zero maps. In particular, we have $e(\theta(a){\sbt}\theta(b))=\theta(1_A)\theta(ab)=\theta(a)\theta(b)$ for all $\theta(a),\theta(b)\in X$. Hence $(\theta(A),\theta(1_A))$ also satisfy the criteria of \cref{theorem57}.
\end{examples}

\begin{proposition}\label{prop52'}
Let $B$ be a $\Cstar$-algebra and $X\subset B$ a self-adjoint linear subspace with local order unit $e\in B_+^1\cap X'$ with $X^2\subset eX$. 
Then we have the following for all $x,y\in X$ with respect to the product $\sbt$ from \cref{theorem57}. 
\begin{enumerate}[label=\textnormal{(\roman*)}]
    \item $xy=0$ if and only if $x{\sbt} y=0$, \label{prop52' it2}
    \item $xy=yx$ if and only if $x{\sbt} y = y{\sbt} x$, and \label{prop52' it3}
    \item $x^*{\sbt} x\in B_+$. \label{prop52' it4}
\end{enumerate}
\end{proposition}

\begin{proof}
Set $x,y\in X$. By nondegeneracy (\cref{order unit}(i)) we have that 
$$0=xy=e(x{\sbt} y) \iff x{\sbt} y=0.$$
If $xy = yx$, then 
$$e(x{\sbt} y)=xy=yx=e(y{\sbt} x)\ \Longrightarrow\ x {\sbt} y = y {\sbt} x,$$
and likewise $x {\sbt} y = y {\sbt} x$ implies $xy=yx$. 
Finally, $e(x^*{\sbt} x)=x^*x\geq 0$ implies $x^*{\sbt} x\geq 0$ by \cref{order unit}(iv). 
\end{proof}

\begin{remark}\label{X has positive elements}
In particular, this tells us that, even when $e\notin X$, we are still guaranteed to have $X\cap B_+\neq \{0\}$ when $X\neq \{0\}$, which resolves the concern raised in Remark \ref{automatically mou}(ii). 
\end{remark}

We close by confirming that our construction is compatible with matrix amplifications.

\begin{corollary}\label{theorem 57 amplified}
Let $B$ be a $\Cstar$-algebra and $X\subset B$ a self-adjoint linear subspace with local order unit $e\in B_+^1\cap X'$ with $X^2\subset eX$.
Then for every $r\geq 1$, $e^{(r)}\in \M_r(B)_+^1\cap \M_r(X)'$ is a local order unit for $\M_r(X)$ with $\M_r(X)^2\subset e^{(r)}\M_r(X)$, where $e^{(r)}\coloneqq 1_{\M_r}\otimes e$. The $\Cstar$-algebra associated to $(\M_r(X),e^{(r)})$ is $\M_r\big(\Cstar_{\sbt}(X)\big)$. 
\end{corollary}

\begin{proof}
Fix $r\geq 1$.
By Remark \ref{automatically mou}(i) $e$ is automatically a matrix local order unit. Moreover, $e^{(r)}\in \M_r(X)'$, and we have the inclusions 
\[\M_r(X)^2\subset \Span\{\M_r(X^2)\}\subset \M_r(eX)= e^{(r)}\M_r(X).\]
Applying \cref{theorem57} gives us an associative bilinear map $\star\colon \M_r(X)\times \M_r(X)\longrightarrow \M_r(X)$ satisfying
\begin{align*}
    xy=e^{(r)}(x{\star} y)
\end{align*}
for all $x,y\in \M_r(X)$, and a pre-$\Cstar$-norm $\|\cdot\|_{\star}$ on $(\M_r(X),{\star})$ defined for each $x\in \M_r(X)$ by 
 \begin{align*}
 \|x\|_{\star}\coloneqq\lim_j\|(1_{\M_r}\otimes h_j(e))x\|=\lim_j\|h_j(1_{\M_r}\otimes e)x\|=\lim_j\|h_j(e^{(r)})x\|,
 \end{align*}
 where $\{h_j\}_j$ is defined as in \eqref{hk}.
 
For $x=(x_{kl})_{k,l=1}^r$ and $y=(y_{kl})_{k,l=1}^r\in \M_r(X)$,
\begin{align*}
    e^{(r)} (x{\star} y)&=xy\\
    &=\big(\textstyle{\sum}_{m=1}^r x_{km}y_{ml}\big)_{k,l=1}^r\\
    &=\big(\textstyle{\sum}_{m=1}^r e (x_{km}{\sbt} y_{ml})\big)_{k,l=1}^r\\
    &=e^{(r)}\big(\textstyle{\sum}_{m=1}^r(x_{km}{\sbt} y_{ml})\big)_{k,l=1}^r,
\end{align*}
and so nondegeneracy (\cref{order unit}(i)) tells us %\cref{order unit'} (1) tells us 
that 
$$x{\star} y=\big(\textstyle{\sum}_{m=1}^n (x_{km}{\sbt} y_{ml})\big)_{k.l=1}^r,$$
whence we may write $\sbt$ for both products.

Next, since for each $j\geq 1$ and $x=(x_{kl})_{k,l=1}^r\in \M_r(X)$
\[\max_{1\leq k,l\leq r} \|h_j(e)x_{kl}\|\leq \|(h_j(e)x_{kl})_{k,l=1}^r\|\leq r\max_{1\leq k,l\leq r} \|h_j(e)x_{kl}\|, \]
it follows that 
\begin{align*}
\max_{1\leq k,l\leq r} \|x_{kl}\|_{\sbt}\leq \|x\|_{\star}\leq r\max_{1\leq k,l\leq r} \|x_{kl}\|_{\sbt}.
\end{align*}
Hence we have
\[\overline{\M_r(X)}^{\|\cdot\|_{\sbt}}=\M_r(\Cstar_{\sbt}(X)).\qedhere\]
\end{proof}

With this, we may write $\Cstar_{\sbt}\big(\M_r(X)\big)=\M_r\big(\Cstar_{\sbt}(X)\big)$ for all $r\geq 1$ without ambiguity. 
%---------------------------------------------------------------------------------------------------------

\section{A c.p.c.\ order zero map}\label{sect: a c.p.c. order zero map}

\noindent We now turn our attention to the map 
$\id_X:(X,\|\cdot\|_{\sbt})\longrightarrow (X,\|\cdot\|)$ between the normed linear spaces $(X,\|\cdot\|_{\sbt})$ and $(X,\|\cdot\|)$ (where $\|\cdot\|_{\sbt}$ is as defined in \eqref{bullet norm} and $\|\cdot\|$ is the norm inherited from the $\Cstar$-algebra $B$ as before), which we want to extend to a c.p.c.\ order zero map $\Cstar_{\sbt}(X)\longrightarrow B$. The map $\id_X$ is already a linear $^*$-preserving bijection, but in order to extend it to $\Cstar_{\sbt}(X)$, we need to show that it is bounded. In fact, it is completely contractive as the following proposition shows.

\begin{proposition}\label{equivalent norms}
Let $B$ be a $\Cstar$-algebra and $X\subset B$ a self-adjoint linear subspace with local order unit $e\in B_+^1\cap X'$, and let $\|\cdot\|_{\sbt}$ be as defined in \eqref{bullet norm}.
Then for each $r\geq 1$ and $x\in \M_r(X)$, $$\|x\|\leq \|x\|_{\sbt}.$$ 

Moreover, if we assume that $X$ is closed in $\|\cdot\|$, then there exists an $R\geq 1$ such that for all $x\in X$, 
$$\|x\|_{\sbt}\leq R\|x\|.$$ In particular, if $X=\overline{X}^{\|\cdot\|}$, then $X=\overline{X}^{\|\cdot\|_{\sbt}}$. 
\end{proposition}

\begin{proof}
Let $\{h_j\}_j$ be as defined \eqref{hk}, and fix $r\geq 1$ and $x\in \M_r(X)$. 
Since $e^{(r)}$ is still a local order unit for $\M_r(X)$, $\{h_j(e^{(r)})e^{(r)}\}_j\subset X'$ forms an increasing approximate identity for $X$ by \cref{commuting order unit}(iii). Combining this with \cref{theorem 57 amplified} and \eqref{bullet norm 2}, we have 
\begin{align}\label{contractive}
    \|x\|&=\lim_j\|h_j(e^{(r)})e^{(r)}x\|\\
    &=\sup_j\|h_j(e^{(r)})e^{(r)}x\|\nonumber\\
    &\leq \sup_j\|e^{(r)}\|\|h_j(e^{(r)})x\| \nonumber\\
    &\leq\|e^{(r)}\|\|x\|_{\sbt}\nonumber\\
    &\leq \|x\|_{\sbt}. \nonumber
\end{align}

If $X$ is closed in $\|\cdot\|$, then $X$ is a Banach space with respect to $\|\cdot\|$, and the maps $H_j\colon X\longrightarrow B$ given for each $j\geq 1$ for all $x\in X$ by 
\begin{align}\label{Hj}
    H_j(x)=h_j(e)x
\end{align} 
are bounded linear maps.  Since $\sup_j\|H_j(x)\|=\|x\|_{\sbt}<\infty$ for each $x\in X$ (by \cref{norm}), the Banach--Steinhaus theorem (or Uniform Boundedness Principle)
guarantees that $0\leq \sup_j\|H_j\|=:R<\infty$. 
Thus we have for all $x\in X$
\begin{align}\label{Hj'}
\|x\|_{\sbt}&=\lim_j\|h_j(e)x\|\\
&=\sup_j\|h_j(e)x\|\nonumber\\
&=\sup_j\|H_j(x)\|\nonumber\\
&\leq \sup_j\|H_j\|\|x\|\nonumber\\
&= R\|x\|. \nonumber
\end{align}
Moreover, \eqref{contractive} tells us $R\geq 1$. 
\end{proof}

\begin{theorem}\label{prop 61}
Let $B$ be a $\Cstar$-algebra and $X\subset B$ a self-adjoint linear subspace with local order unit $e\in B_+^1 \cap X'$ with $X^2\subset eX$. Let $\Cstar_{\sbt}(X)$ be the associated $\Cstar$-algebra as in Definition \ref{defn60}. 
Then the map 
$\id_X:(X,\|\cdot\|_{\sbt})\longrightarrow (X,\|\cdot\|)$ extends to a c.p.c.\ order zero map
$\Phi\colon \Cstar_{\sbt}(X) \longrightarrow B$ with $\Phi(\Cstar_{\sbt}(X))\subset \overline{X}^{\|\cdot\|}$.  

If $\Phi$ is injective, then $\Phi^{(r)}(x)\geq 0$ if and only if $x\geq 0$ for any $r\geq 1$ and $x\in \M_r(\Cstar_{\sbt}(X))$. 
 Moreover $X$ is closed in $\|\cdot\|$ if and only if $\Phi=\id_X$ and $\Phi$ is bounded below.
\end{theorem}

\begin{proof}
It follows from \cref{equivalent norms} that the map $\Cstar_{\sbt}(X)\supset X\xlongrightarrow {\id_X} X\subset B$ is completely contractive and hence extends to a completely contractive map $\Phi\colon \Cstar_{\sbt}(X)\longrightarrow B$ with $\Phi(\Cstar_{\sbt}(X))\subset \overline{X}^{\|\cdot\|}$. To see that $\Phi$ is completely positive, fix $r\geq 1$. 
By \cref{theorem 57 amplified}, 
we may apply \cref{prop52'}\ref{prop52' it4}
and conclude that for each $x\in \M_r(X)$, 
$$\Phi^{(r)}(x^*{\sbt} x)=x^*{\sbt} x\in \M_r(B)_+.$$
Since $(\M_r(X),{\sbt})$ is a dense $^*$-subalgebra of $\M_r(\Cstar_{\sbt}(X))=\Cstar_{\sbt}(\M_r(X))$, it follows that $\{x^*{\sbt} x\ :\ x\in \M_r(X)\}$ is dense in $\M_r(\Cstar_{\sbt}(X))_+$, and so
we have 
\begin{align*}
\Phi^{(r)}(\M_r(\Cstar_{\sbt}(X))_+)&\subset \overline{\Phi^{(r)}(\{x^*{\sbt} x\ :\ x\in \M_r(X)\})}^{\|\cdot\|}\\
&=\overline{\{x^*{\sbt} x\ :\ x\in \M_r(X)\}}^{\|\cdot\|}\subset \M_r(B)_+.
\end{align*}
To see that $\Phi$ is order zero, suppose $a,b\in \Cstar_{\sbt}(X)_+$ with $a{\sbt} b=0$, and let $(x_n)_n,(y_n)_n\subset X$ such that $\lim_n\|x_n- a\|_{\sbt}=0$ and $\lim_n\|y_n- b\|_{\sbt}=0$. Then $\lim_n \|x_n{\sbt} y_n\|_{\sbt}=0$ and so
\begin{align*}
0&=\lim_n \|e\Phi(x_n{\sbt} y_n)\|\\&=\lim_n \|e(x_n{\sbt} y_n)\|\\&=\lim_n \|x_ny_n\|\\
&=\lim_n \|\Phi(x_n)\Phi(y_n)\|\\&=\|\Phi(a)\Phi(b)\|, 
\end{align*} 
using continuity of $\Phi$ for the first and last equalities. 

 Remark \ref{cor: order zero for unital}(iii) tells us that 
if $\Phi$ is injective, then for any $r\geq 1$ and $x\in \M_r(\Cstar_{\sbt}(X))$, we have $\Phi^{(r)}(x)\geq 0$ if and only if $x\geq 0$.

 If $X$ is closed in $\|\cdot\|$, then \cref{equivalent norms} tells us that $X$ is already complete with respect to $\|\cdot\|_{\sbt}$. Hence $\Cstar_{\sbt}(X)=X$ as sets, and $\Phi=\id_X$. In particular, $\Phi$ is injective with closed range and hence bounded below. On the other hand, if $\Phi$ is bounded below, then its image is closed in $B$, and so $\Phi(\Cstar_{\sbt}(X))= \overline{X}^{\|\cdot\|}$. If moreover $\Phi=\id_X$, then it follows that $X$ is closed in $B$.  
\end{proof}

Recall from Example \ref{prop52.5}(ii) that for a given a c.p.c.\ order zero map $\theta\colon A\longrightarrow B$ from a unital $\Cstar$-algebra, $\theta(A)$ and $\theta(1_A)$ satisfy the assumptions of \cref{theorem57}. 
In this case, it turns out that $\theta(A)$ is automatically closed with respect to the induced $\|\cdot\|_{\sbt}$ norm. This is essentially \cref{bdd below perp gives isoclass} in our $\sbt$-language. 

\begin{proposition}\label{X=Xbullet}
Suppose $\theta\colon A\longrightarrow B$ is a c.p.c.\ order zero map between $\Cstar$-algebras and there exists a local order unit $e\in B_+^1\cap \theta(A)'$ for $\theta(A)$ such that $\theta(a)\theta(b)=e\theta(ab)$ for all $a,b\in A$. 
Then $\theta(ab)=\theta(a){\sbt}\theta(b)$ for all $a,b\in A$, and $\eta\coloneqq\id_{\theta(A)}\circ \theta\colon A\longrightarrow \Cstar_{\sbt}(\theta(A))$ is a surjective $^*$-homomorphism which is a $^*$-isomorphism when $\theta$ is injective. 
In particular, $\theta(A)$ is closed with respect to $\|\cdot\|_{\sbt}$. 
\end{proposition}

\begin{proof}
Since $(\theta(A),e)$ satisfy the assumptions of \cref{theorem57} (see Example \ref{prop52.5}(ii)) we have a product $\sbt$ on $\theta(A)$ satisfying $e\theta(ab) = \theta(a)\theta(b)\\ = e(\theta(a){\sbt} \theta(b))$ for all $a,b\in A$.  It then follows from nondegeneracy (\cref{order unit}(i)) that $\theta(ab)=\theta(a){\sbt} \theta(b)$ for all $a,b\in A$. 

Since the map $\id_{\theta(A)}:(\theta(A),\|\cdot\|)\longrightarrow (\theta(A),\|\cdot\|_{\sbt})$ between the normed $^*$-linear spaces 
is linear and $^*$-preserving (though a priori not necessarily bounded), the map $\eta\colon A\longrightarrow \Cstar_{\sbt}(\theta(A))$ is also linear and $^*$-preserving with $\eta(A)=\theta(A)$ dense in $\Cstar_{\sbt}(\theta(A))$. Moreover, for each $a,b\in A$, we have $\eta(ab)=\theta(ab)=\theta(a){\sbt} \theta(b)=\eta(a){\sbt} \eta(b)$. Hence $\eta$ is a surjective $^*$-homomorphism, which implies that its image, $\theta(A)$, is closed in $\Cstar_{\sbt}(\theta(A))$, i.e. $\theta(A)=\overline{\theta(A)}^{\|\cdot\|_{\sbt}}$. When $\theta$ is moreover injective, so is $\eta$. 
\end{proof}

\begin{remark}
It follows from the (proof of the) structure theorem for c.p.c.\ order zero maps (\cite[Theorem 3.3]{WZ09}) that the kernel of any c.p.c.\ order zero map $\theta\colon A\longrightarrow B$ from a unital $\Cstar$-algebra is an ideal. In this case, $\Cstar_{\sbt}(\theta(A))$ is isomorphic to $A/\ker({\theta})$, which is in turn isomorphic to $\pi_\theta(A)$ where $\pi_\theta\colon A\longrightarrow \mathcal{M}(\Cstar(\theta(A)))$ is the $^*$-homomorphism guaranteed by the structure theorem. 
\end{remark}

We are now equipped to say moreover when a self-adjoint linear subspace of a $\Cstar$-algebra is actually the image of a c.p.c.\ order zero map from a unital $\Cstar$-algebra. 

\begin{proposition}\label{image of c.p.c. order zero}
Let $B$ be a $\Cstar$-algebra with $e\in B_+^1$ and $X\subset B$ a self-adjoint linear subspace. Then the following are equivalent.
\begin{enumerate}[label=\textnormal{(\roman*)}]
    \item There exists a unital $\Cstar$-algebra $A$ and c.p.c.\ order zero map $\theta\colon A\longrightarrow B$ with $\theta(A)=X$ and $\theta(1_A)=e$. 
    \item $e\in X\cap X'$ is a local order unit for $X$, $X^2= eX$, and $X=\overline{X}^{\|\cdot\|_{\sbt}}$.
\end{enumerate}
\end{proposition}

\begin{proof}
Example \ref{prop52.5}(ii) and \cref{X=Xbullet} give (i) $\Longrightarrow$ (ii).

On the other hand, if $e\in X\cap X'$ is a local order unit such that $X^2= eX$ and $X=\overline{X}^{\|\cdot\|_{\sbt}}$, then \cref{theorem57} gives us a unital $\Cstar$-algebra $\Cstar_{\sbt}(X)$, and \cref{prop 61} gives us $\Phi=\id_X\colon \Cstar_{\sbt}(X)\longrightarrow B$, the desired c.p.c.\ order zero map. 
\end{proof}

%-------------------------------------------------------------------------------------------------------------------

\section{Unitizations}\label{Z}

\noindent In this section, we describe a unitization of our construction and consider the dependence of the construction on the given local order unit. This will also yield the proof of \cref{theorem C} (by combining Proposition \ref{theorem57 span} with Remarks \ref{rem: unitization}).

\begin{proposition}\label{theorem57 span}
Let $B$ be a $\Cstar$-algebra and $X\subset B$ a self-adjoint linear subspace, and suppose there exists a local order unit $e\in B_+^1\cap X'$ for $X$ so that $X^2\subset eX$. 
Set $Z\coloneqq\Span\{X,e\}$. 

Then $e\in B_+^1\cap Z'$ is a local order unit for $Z$ so that $Z^2=eZ$, and the induced product and norm from \emph{\cref{theorem57}} agree with those on $X$. 

Moreover, the inclusion $X\subset Z$ induces an inclusion of $\Cstar_{\sbt}(X)$ in $\Cstar_{\sbt}(Z)$ as a two-sided closed ideal, which has co-dimension 1 unless $e\in X$, and 
the map $\id_Z$  extends to a c.p.c.\ order zero map $\Phi^\dagger\colon \Cstar_{\sbt}(Z)\longrightarrow  B$ which agrees with $\id_Z$ on $Z$. 
\end{proposition}

\begin{proof}
If $e\in X$, then there is nothing to show, so we assume $e\notin X$. We want to apply \cref{theorem57} to $Z$. One readily verifies that $e\in Z'\cap B_+^1$ is a local order unit for $Z$. 
For $\alpha, \beta\in \C$ and $x,y\in X$, let $z\in X$ so that $xy=ez$. Then 
\begin{align*}
    (\alpha e+x)(\beta e+y)&=\alpha\beta e^2 +\alpha ey +\beta ex+ xy\\
    &=\alpha\beta e^2 +\alpha ey +\beta ex+ ez\in eZ. %\\
    %&=e(\alpha\beta e +\alpha y +\beta x+ z)\in eZ.
\end{align*}
Hence $Z^2=eZ$. Applying \cref{theorem57} to both $(X,e)$ and $(Z,e)$ gives us 
associative bilinear maps ${\sbt}\colon X\times X\longrightarrow X$ and $\star\colon Z\times Z\longrightarrow Z$ such that for all $x,y\in X$ and $w,z\in Z$, 
\begin{align*}
    xy=e(x{\sbt} y) \hspace{.5 cm} \text{ and } \hspace{.5 cm} wz=e(w\star z).
\end{align*}
Nondegeneracy (\cref{order unit}(i)) for $(Z,e)$ tells us that for all $x,y\in X\subset Z$, we have
\[e(x\star y)=xy=e(x{\sbt} y)\ \Longrightarrow x\star y = x{\sbt} y.\] 
Then $\star\colon Z\times Z\longrightarrow Z$ extends ${\sbt}\colon X\times X\longrightarrow X$, and we may write ${\sbt}$ for the $\star$ multiplication on $Z$ without confusion. 

By definition the norm $\|\cdot\|_{\sbt}\colon Z\longrightarrow [0,\infty)$ guaranteed by Theorem \ref{theorem57} restricts to the norm $\|\cdot\|_{\sbt}\colon X\longrightarrow [0,\infty)$ guaranteed by Theorem \ref{theorem57}, and hence %and by \cref{B-space closures}, 
$\Cstar_{\sbt}(Z)=\overline{\Span\{X,e\}}^{\|\cdot\|_{_{\sbt}}}=\Span\{\overline{X}^{\|\cdot\|_{_{\sbt}}},e\}=\Span\{\Cstar_{\sbt}(X),e\}$.\footnote{As a consequence of the Banach space fact from footnote \ref{B-space closures1}, it follows that for any linear subspaces $\mathcal{Y}$ and $\mathcal{Z}$ of a Banach space $\mathcal{X}$ with  $\mathcal{Z}$ finite-dimensional, $\overline{\Span\{\mathcal{Y},\mathcal{Z}\}}=\Span\{\overline{\mathcal{X}},\mathcal{Z}\}$.}\label{B-space closures2} 
Hence the inclusion $X\subset Z$ induces an isometric embedding $\Cstar_{\sbt}(X)\subset \Cstar_{\sbt}(Z)$. %, where $\Cstar_{\sbt}(Z)$ is generated by $\Cstar_{\sbt}(X)$ and $e$. 
Since $e$ is the unit of $\Cstar_{\sbt}(Z)$ (by \cref{theorem57}), it follows that $\Cstar_{\sbt}(X)$ sits as a closed 2-sided ideal with co-dimension at most 1. 
Note that if $\Cstar_{\sbt}(X)=\Cstar_{\sbt}(Z)$, then $(X,{\sbt})\subset \Cstar_{\sbt}(Z)$ is a dense 2-sided ideal, which must then contain the unit $e$. 
So $\Cstar_{\sbt}(X)$ has co-dimension 1 if and only if $e\notin X$.

It follows from \cref{prop 61} that $\id_Z$ extends to a c.p.c.\ order zero map $\Phi^\dagger\colon \Cstar_{\sbt}(Z)\longrightarrow B$, and %\cref{B-space closures} tells us that 
$X\subset \Cstar_{\sbt}(Z)$ is closed with respect to $\|\cdot\|_{\sbt}$ if and only if $Z$ is.\footnote{See footnote \ref{B-space closures1}.}
\end{proof}

\begin{remarks}\label{rem: unitization}
(i) In this generality, there is nothing ruling out the possibility that $e\notin X$ but $\Cstar_{\sbt}(X)$ is nonetheless unital, in which case $\Cstar_{\sbt}(Z)$ would end up being the forced unitization $\Cstar_{\sbt}(X)^\dagger$ of $\Cstar_{\sbt}(X)$. (Indeed a unital $\Cstar$-algebra sitting inside its forced unitization is already an example of this.) 
However, note that $\Cstar_{\sbt}(X)\triangleleft \Cstar_{\sbt}(Z)$ is an essential ideal if and only if $\Cstar_{\sbt}(X)$ is not unital. Hence, when we really want to deal with
$\Cstar_{\sbt}(X)^\sim$, 
we use $\Cstar_{\sbt}(Z)$ when 
$\Cstar_{\sbt}(X)\triangleleft \Cstar_{\sbt}(Z)$ is essential and 
$\Cstar_{\sbt}(X)^\sim$ otherwise. In this case, we write $\Phi^{\sim}$ for $\Phi^\dagger$. 

(ii) As a Banach space, 
$Z$ is closed with respect to $\|\cdot\|_{\sbt}$ if and only if $X$ is.\footnote{See footnote \ref{B-space closures1}.} Likewise, when $\Cstar_{\sbt}(Z)=\Cstar_{\sbt}(X)^\sim$, $\Phi^{\sim}$ is bounded below if and only if $\Phi$ is. Indeed, since $\Cstar_{\sbt}(Z)=\Span\{\Cstar_{\sbt}(X),e\}$ and $\Phi$ is linear, it follows that $\Phi^{\sim}\big(\Cstar_{\sbt}(Z)\big)=\Span\{\Phi\big(\Cstar_{\sbt}(X)\big),e\}$, and so $\Phi$ has closed image if and only if $\Phi^{\sim}$ does. Clearly $\Phi$ is injective if $\Phi^{\sim}$ is, and the other direction we checked in 
\cref{invertible c.p. order zero and order scale}(ii). 
\end{remarks}

Now, \cref{theorem C} follows from \cref{theorem57 span} and  \cref{rem: unitization}. 

As a corollary to \cref{theorem57 span}, we get sufficient conditions for when a c.p.c.\ order zero map $A\longrightarrow B$ from a non-unital $\Cstar$-algebra extends to a c.p.c.\ order zero map $A^{\sim}\longrightarrow B$.  

\begin{corollary}\label{unitize order zero'}
Let $A$ be a non-unital $\Cstar$-algebra and $\theta\colon A\longrightarrow B$ a c.p.c.\ order zero map. If there exists a local order unit $e\in B_+^1\cap \theta(A)'$ for $\theta(A)$ such that $e\theta(ab)=\theta(a)\theta(b)$ for all $a,b\in A$, then $\theta^{\sim}\colon A^{\sim}\longrightarrow B$ given by $\theta^{\sim}(a+\lambda1_{A^{\sim}})=\theta(a)+\lambda e$ is a c.p.c.\ order zero extension of $\theta$.
\end{corollary}

\begin{proof}
Set $X=\theta(A)$ and $Z=\Span\{X,e\}$. \cref{X=Xbullet} tells us  $X=\Cstar_{\sbt}(X)$ as sets, and so 
 by Remark \ref{rem: unitization}(ii) $Z=\Cstar_{\sbt}(Z)$ as well,  and  
 the c.p.c.\ order zero map $\Phi^\dagger\colon \Cstar_{\sbt}(Z)\longrightarrow B$ from \cref{theorem57 span} is just $\id_Z$. 

Let $\eta\coloneqq\id_X\circ\theta\colon A\longrightarrow \Cstar_{\sbt}(X)$ be the $^*$-homomorphism from Proposition \ref{X=Xbullet}. Then $\eta$ extends to a unital $^*$-homomorphism $\eta^\sim\colon A^{\sim}\longrightarrow \Cstar_{\sbt}(Z)$ given by $\eta^\sim(a+\lambda 1_{A^{\sim}})=\eta(a) +\lambda e= \theta(a) +\lambda e$. Hence $\theta^{\sim}\coloneqq\Phi^\dagger\circ\eta^\sim$ is a c.p.c.\ order zero map, %(as the composition of a c.p.c.\ order zero map with a $^*$-homomorphism), 
and it is given by 
$\theta^{\sim}(a+\lambda1_{A^{\sim}})=\theta(a)+\lambda e$, as desired. 
\end{proof}

For the remainder of this section, we consider how our construction of $\Cstar_{\sbt}(X)$ depends on the distinguished  commuting local order unit $e$. The following is essentially what is employed in \cite[Proposition 3.5]{CW1}. 

\begin{proposition}\label{mult is order unit}
Let $B$ be a $\Cstar$-algebra and $X\subset B$ a self-adjoint linear subspace with local order unit $e\in B_+^1\cap X'$ such that $X^2\subset eX$. Suppose $h\in B_+^1$ satisfies $hx=ex$ for all $x\in X$. Then $h\in X'$ and $h$ is a local order unit for $X$ such that $X^2\subset hX$. If $e$ is moreover 
an        order unit
then $h$ is also an        order unit 
with the same  
scaling factor as $e$. Moreover, the products and norms induced by $e$ and $h$ on $X$ agree, and the induced $\Cstar$-algebras are equal.
\end{proposition}

\begin{proof}
For $x,y\in X$, there exists $z\in X$ so that $xy=ez=hz$, and so it follows that $X^2\subset hX$. Moreover, 
\[\|hx-xh\|=\|ex-xh\|=\|xe-xh\|=\|ex^*-hx^*\|=0,\]
and so $h\in X'$. 

In order to check that $h$ is a local order unit, first note that for all $a\in \Cstar(X)\subset B$ (the $\Cstar$-algebra generated by $X$ in $B$) we also have $ea=ha$. In particular, $e|x|=h|x|$ for all $x\in X$. Let $x\in X$ be self-adjoint. Then \cref{commuting order unit}\ref{commuting order unit (b)} guarantees an $R>0$ so that $Re\geq |x|$, which by \cref{commuting order unit}\ref{eq: hashtag} implies $Re|x|\geq x^*x$. Since $Re|x|=Rh|x|$, it follows that $Rh|x|\geq x^*x$, which by \cref{commuting order unit}\ref{eq: hashtag} implies $Rh\geq |x|$. Then \cref{commuting order unit}\ref{commuting order unit (b)} tells us $h$ is a local order unit for $X$.

If $e$ is moreover an        order unit with scaling factor $R$, then 
the same argument shows that $h$ is an        order unit with scaling factor at most as large as $R$. A symmetric argument now shows that $R$ is the scaling factor of $h$. 

Finally, it follows from nondegeneracy and the fact that $h_j(e)x=h_j(h)x$ for all $x\in X$ and $j\geq 1$ that the $\Cstar$-algebras associated to $(X,e)$ and $(X,h)$, as guaranteed by \cref{theorem57}, are equal. 
\end{proof}

The previous lemma applies in particular to the case where $\Cstar_{\sbt}(X)$ is unital but does not contain $e$, as in the following corollary.  
\begin{corollary}\label{mult is order unit'} 
Let $B$ be a $\Cstar$-algebra and $X\subset B$ a self-adjoint linear subspace with         order unit 
$e\in B_+^1\cap X'$ such that $X^2\subset eX$, and let $\Cstar_{\sbt}(X)$ denote the associated $\Cstar$-algebra. If $\Cstar_{\sbt}(X)$ is unital but $e\notin X$, then the unit of $\Cstar_{\sbt}(X)$, denoted by $h$, is also 
an         order unit for $X$ 
with the same scaling factor as $e$, and the $\Cstar$-algebra associated to $(X,h)$ is exactly $\Cstar_{\sbt}(X)$. 
\end{corollary}

\begin{proof}
     Let $x\in X$. Since $\Cstar_{\sbt}(X)$ is a $\Cstar$-algebra, there exist $y,z\in \Cstar_{\sbt}(X)$ so that $x=y\sbt z$. Then since $\Phi$ is order zero, we have \[hx=\Phi(h)\Phi(y\sbt z)=\Phi(y)\Phi(z)=yz=e(y\sbt z)=ex.\]
      By \cref{mult is order unit} we are done.
\end{proof}

 \begin{proposition}\label{order unit invariant}
 Let $X\subset B$ be a self-adjoint linear subspace of a $\Cstar$-algebra $B$ with local order units $e,h\in B_+^1\cap X\cap X'$ with $hX=X^2=eX$. 
 Let $\sbt \colon X\times X\longrightarrow X$, $\|\cdot\|_{\sbt }$, $\Cstar_{\sbt }(X)$ denote the multiplication, norm, 
 and $\Cstar$-algebra induced by $e$ and $\star \colon X\times X\longrightarrow X$, $\|\cdot\|_{\star }$, and $\Cstar_{\star }(X)$ the multiplication, norm, and $\Cstar$-algebra induced by $h$. If $X$ is complete with respect to both  $\|\cdot\|_{\star }$ and $\|\cdot\|_{\sbt}$, then $\Cstar_{\sbt }(X)$ and $\Cstar_{\star }(X)$ are unitally $^*$-isomorphic. 
 \end{proposition}
 
 \begin{proof}
First, we remark that our assumptions guarantee that $\Cstar_{\sbt }(X)=X=\Cstar_{\star }(X)$ as $^*$-vector spaces, and $e$, resp.\ $h$, is the unit of $\Cstar_{\sbt }(X)$, resp.\ $\Cstar_{\star }(X)$. We also note that $e\in Z(\Cstar_{\star }(X))$ and $h\in Z(\Cstar_{\sbt }(X))$ by \cref{prop52'}(ii). By \cref{prop 61}, for all $x\in X$ \[x\in \Cstar_{\sbt }(X)_+\ \Longleftrightarrow\  x\in B_+\ \Longleftrightarrow\  x\in \Cstar_{\star }(X)_+.\]
So $Rh- e\in \Cstar_{\sbt }(X)_+$ and $Re- h\in \Cstar_{\star }(X)_+$ for some $R> 0$, which means $h\in \mathrm{GL}(\Cstar_{\sbt }(X))$ and $e\in \mathrm{GL}(\Cstar_{\star }(X))$. Let $h^{-1}$ denote the element of $X$ that acts as the inverse of $h$ in $\Cstar_{\sbt }(X)$, i.e., $h\sbt  h^{-1}=h^{-1}\sbt  h = e$, and let $e^{-1}$ likewise denote the inverse of $e$ in $\Cstar_{\star }(X)$. Then for all $x,y\in X$, we have 
\begin{align*}
    h(x\star  y) &= xy\\
    &= e(x\sbt  y)\\ &= e(h \sbt  h^{-1}\sbt x\sbt  y)\\ &= e(h \sbt  (h^{-1}\sbt x\sbt  y))\\ &= h(h^{-1}\sbt x\sbt  y).
\end{align*}
Nondegeneracy of $h$ implies that 
\begin{align}\label{prod1}
x\star  y = h^{-1}\sbt x\sbt  y
\end{align}
for all $x,y\in X$, and likewise, we have 
\begin{align}\label{prod2}
x\sbt  y= e^{-1}\star x\star  y
\end{align}
for all $x,y\in X$. 
In particular, we have for all $x,y\in X$ 
\begin{align}
    h &= e\star  e^{-1} = h^{-1}\sbt e\sbt  e^{-1} = h^{-1}\sbt e^{-1},\ \mathrm{ and } \label{inverses1}\\
    e &= h\sbt  h^{-1} = e^{-1}\star h\star  h^{-1}  = e^{-1}\star h^{-1}. \label{inverses2}
\end{align}
Now define $\varphi\colon \Cstar_{\star }(X)\longrightarrow \Cstar_{\sbt }(X)$ by $\varphi(x)=h^{-1}\sbt x$ for all $x\in X$. Then $\varphi$ is a unital $^*$-homomorphism. Indeed, $\varphi(h)=h^{-1}\sbt h= e$, and for $x,y\in X$, we have 
\begin{align*}
    \varphi(x\star  y)= h^{-1}\sbt (x\star  y)\overset{\eqref{prod1}}{=} h^{-1}\sbt (h^{-1}\sbt x\sbt  y) = (h^{-1}\sbt x)\sbt  (h^{-1}\sbt y).
\end{align*} 
Likewise, $\psi\colon \Cstar_{\sbt }(X)\longrightarrow \Cstar_{\star }(X)$ given by $\psi(x)= e^{-1}\star  x $ is a unital $^*$-homomorphism. Moreover, %using \eqref{inverses1} and \eqref{inverses2}, 
we have for any $x\in X$, 
\begin{align*}
    \psi(\varphi(x))&=e^{-1}\star (h^{-1}\sbt x)\overset{\eqref{prod2}}{=}e^{-1}\star (e^{-1}\star h^{-1}\star  x) \overset{\eqref{inverses2}}{=} e^{-1}\star (e\star x) = x.
\end{align*}
Likewise using \eqref{inverses1}, we have $\varphi\circ\psi(x)=x$. 
So $\psi=\varphi^{-1}$. 
 \end{proof}

\begin{example}
    Let $A,B,$ and $C$ be $\Cstar$-algebras with $A$ and $C$ unital. Suppose we have c.p.\ order zero maps $\varphi\colon A \longrightarrow B$ and $\psi\colon C \longrightarrow B$ with $\varphi(A)=\psi(C)$. Then $X=\varphi(A)=\psi(C)$ and $e=\varphi(1_A)$ and $h=\psi(1_C)$ satisfy the assumptions of \cref{order unit invariant}, and so, with the notation above,  $\Cstar_{\sbt}(\varphi(A))\cong \Cstar_{\star}(\psi(C))$. Note that it now follows from \cref{X=Xbullet} that this $\Cstar$-algebra is a quotient of both $A$ and $C$. 
\end{example}

%-------------------------------------------------------------------------

\section{Closed subspaces and order units}\label{closure}

\noindent The theory is better behaved when our self-adjoint linear subspace $X\subset B$ is assumed to be closed and also when the  commuting local order unit for $X$ is assumed to be an        order unit. It turns out these two notions are deeply intertwined.

\begin{proposition}\label{order unit for closure}
Suppose $X\subset B$ is a self-adjoint linear subspace of a $\Cstar$-algebra with         order unit %uniform %(matrix)  order unit 
$e\in B_+^1\cap X'$ with scaling factor $R$. 
Writing $\overline{X}$ for the closure of $X$ with respect to $\|\cdot\|$, we have the following. 
\begin{enumerate}[label=\textnormal{(\roman*)}]
\item \label{it1:order unit for closure} $e$ is an        order unit for $\overline{X}$ with scaling factor $R$.
    \item \label{it2:order unit for closure}  
    $R\|ex\|\geq \|x\|$ for each self-adjoint $x\in X$. In particular, multiplication by $e$ is bounded below on $\overline{X}$, and $\overline{eX}=e\overline{X}$.
    \item\label{it3:order unit for closure}   
    If $X^2\subset \overline{eX}$, then 
    $(\overline{X}, e)$ satisfies the assumptions of Theorem \ref{theorem57}, where the induced norm and product agree with those on $X$, and we have  $\overline{X}=\overline{X}^{\|\cdot\|_{\sbt}}$ and $\Cstar_{\sbt}(X)=\Cstar_{\sbt}(\overline{X})$. 
\end{enumerate}
 
\end{proposition}

\begin{proof}
(i) Suppose $x\in \overline{X}$ is self-adjoint and $(x_n)_n$  is a sequence of self-adjoint elements in $X$ with $\lim_n\|x_n-x\|=0$.
Then $R\|x_n\|e+x_n\geq 0$ for all $n\in \N$, 
which implies $R\|x\|e+x\geq 0$ since $B_+$ is closed. Since $X\subset \overline{X}$, the scaling factor for $e$ with respect to $\overline{X}$ is at least $R$, and so it is $R$.  

(ii) 
Fix $x\in X$, and assume for now that $x$ is self-adjoint. 
Since $e\in X'$, $R\|x\|e\geq x$ implies $R\|x\|e|x|\geq x^2$ by \cref{commuting order unit}\ref{eq: hashtag}, 
and so 
$$R\|x\|\|ex\|=%R\|x\|\||ex|\|=
R\|x\|\|e|x|\|\geq \|x^2\|=\|x\|^2.$$
Now, for general $x\in X$, we have 
\begin{align}\label{for plain x}
    2\|x\|&\leq \|x+x^*\|+\|i(x-x^*)\|\\
    &\leq R\|e(x+x^*)\|+R\|e(i(x-x^*))\|\nonumber\\
    &\leq 4R\|ex\|. \nonumber
\end{align}
It follows that multiplication by $e$ is bounded below on $\overline{X}$, and hence $e\overline{X}=\overline{eX}$. 

(iii) If $X^2\subset \overline{eX}$, then by (ii), we have $(\overline{X})^2\subset \overline{X^2}\subset \overline{eX}=e\overline{X}$, 
and so $(\overline{X}, e)$ satisfies the assumptions of \cref{theorem57}. 
Then we have an associative bilinear map $\overline{X}\times \overline{X}\longrightarrow \overline{X}$ satisfying \eqref{mult id}, which nondegeneracy tells us agrees with ${\sbt}\colon X\times X\longrightarrow X$ on $X\times X$. We also have a $\Cstar$-norm on $\overline{X}$ given by \eqref{bullet norm}, which agrees by definition with $\|\cdot\|_{\sbt}\colon X\longrightarrow [0,\infty)$ on $X$.  
Hence, we can denote the induced multiplication and norm on $\overline{X}$ by ${\sbt}$ and $\|\cdot\|_{\sbt}$, and 
we have a natural inclusion $\Cstar_{\sbt}(X)\subset \Cstar_{\sbt}(\overline{X})$ coming from $(X,{\sbt})\subset (\overline{X},{\sbt})$. 
Since $\overline{X}$ is closed, it then follows from \cref{equivalent norms} that $\|\cdot\|$ and $\|\cdot\|_{\sbt}$ are equivalent norms on $X$, and so 
$\overline{X}=\Cstar_{\sbt}(\overline{X})=\Cstar_{\sbt}(X)$. 
\end{proof}

\begin{remark}
The weaker assumption $X^2\subset \overline{eX}$ may seem out of place, but it will be useful in upcoming work and only requires a minimal adjustment here.\end{remark}

Now we are equipped to explicitly describe the relationship between closed self-adjoint linear subspaces and         order units. 

\begin{corollary}\label{closed uniform order unit}
Let $B$ be a $\Cstar$-algebra and $X\subset B$ a self-adjoint linear subspace with local order unit $e\in B_+^1\cap X'$ such that $X^2\subset eX$. Then $X$ is closed with respect to $\|\cdot\|$ if and only if $e$ is an        order unit and $X$ is closed with respect to $\|\cdot\|_{\sbt}$. 
In particular, if $A$ is a unital $\Cstar$-algebra and $\theta\colon A\longrightarrow B$ a c.p.c.\ order zero map, then $\theta(A)$ is closed if and only if $\theta(1_A)$ is an        order unit for $\theta(A)$. 
\end{corollary}

\begin{proof} 
Set $Z=\Span\{X,e\}$. If $X$ is closed with respect to $\|\cdot\|$, then by footnote \ref{B-space closures1} so is $Z$. Then it follows from \cref{equivalent norms} that %$X=\overline{X}^{\|\cdot\|_{\sbt}}$ and moreover that $Z=\Span\{X,e\}$ is also closed in $B$, or equivalently  
$Z$ is also closed with respect to $\|\cdot\|_{\sbt}$, and  
from \cref{prop 61} that $\Phi^\dagger\colon \Cstar_{\sbt}(Z)\longrightarrow Z$ is bounded below with $Z=\Phi^\dagger(\Cstar_{\sbt}(Z))$. Then as we saw in Example \ref{ex: bdd below UOU}(ii), $e$ is an        order unit for $Z$ in $B$ and hence also for $X$. 
Conversely, if $e$ is an        order unit for $X$ and $X=\overline{X}^{\|\cdot\|_{\sbt}}$, then since $X^2\subset eX\subset \overline{eX}$, it follows from \cref{order unit for closure}(iii) that $\overline{X}^{\|\cdot\|}=\overline{X}^{\|\cdot\|_{\sbt}}=X$. %, which by assumption is equal to $X$. % that $e$ is an        order unit for $\overline{X}$ and $\Cstar_{\sbt}(\overline{X})=\Cstar_{\sbt}(X)=X$. 

When $X=\theta(A)$, \cref{X=Xbullet} says that $X$ is automatically closed with respect to $\|\cdot\|_{\sbt}$. Hence $X$ is closed if and only if $e$ is an order unit. 
\end{proof}

Applying \cref{invertible c.p. order zero and order scale} 
we can now characterize the 
scaling factor for our commuting         order unit 
$e$ as $\|\Phi^{-1}\|$.

\begin{corollary}\label{completely isometric}
Let $B$ be a $\Cstar$-algebra and $X\subset B$ a closed self-adjoint linear subspace with local order unit $e\in B_+^1\cap X'$ such that $X^2\subset eX$. Then $\Phi\colon \Cstar_{\sbt}(X)\longrightarrow B$ is bounded below and $e$ is an        order unit  %for $X$, $X=\Cstar_{\sbt}(X)$ as sets, and $\Phi\coloneqq\id_X\colon \Cstar_{\sbt}(X)$ is a bounded below c.p.\ order zero map. Moreover, 
%$\|\Phi^{-1}\|$ is the scaling factor for $e$.
with scaling factor $\|\Phi^{-1}\|=\sup_j\|H_j\|$, where $H_j\in \mathcal{L}(X,B)$ are defined as in \eqref{Hj}. 
When the scaling factor of $e$ is $1$, $\Phi$ is a complete order embedding and 
 $\|ex\|=\|x\|$ for all $x\in X$. 
\end{corollary}

\begin{proof}
We know from \cref{closed uniform order unit} that $e$ is an        order unit for $X$. 
Set $Z=\text{span}\{X,e\}$. %From \cref{X closed}, we know that $\Phi\colon \Cstar_{\sbt}(X)\longrightarrow  B$ is a bounded below c.p.c.\ order zero map, which extends to the order zero map $\Phi^\dagger\colon \Cstar_{\sbt}(Z)\longrightarrow  B$. 
If $e\in \Cstar_{\sbt}(X)$ or $\Cstar_{\sbt}(X)$ is non-unital, then by \cref{theorem57 span}, $\Cstar_{\sbt}(X)^\sim=\Cstar_{\sbt}(Z)$. Since $X$ is closed in $\|\cdot \|$, so is $Z$,\footnote{See footnote \ref{B-space closures1}.} and so  $\id_Z$ is a bounded below c.p.\ order zero extension of $\Phi$ to $\Cstar_{\sbt}(X)^\sim$ (by \cref{prop 61}). Then %and it follows from \cref{X closed} together with 
\cref{invertible c.p. order zero and order scale}(ii) 
tells us that the 
scaling factor of $e$ for $X$ is  
$\|\Phi^{-1}\|$. 
If $\Cstar_{\sbt}(X)$ is unital with $e\notin\Cstar_{\sbt}(X)$, then \cref{invertible c.p. order zero and order scale}(ii) tells us $\Phi(1_{\Cstar_{\sbt}(X)})$ is an        order unit with scaling factor $\|\Phi^{-1}\|$, and \cref{mult is order unit'}  tells us this is then also the scaling factor for $e$. 
Furthermore, it follows from \eqref{Hj'} that 
for all $x\in X$ with $\|x\|=1$ and all $j\geq 1$ 
\[\|H_j(x)\|\leq \|x\|_{\sbt} =\sup_j\|h_j(e)x\|=\sup_j\|H_j(x)\|\leq \sup_j\|H_j\|<\infty,\]
(where $H_j$ are as defined in \eqref{Hj}), and so for all $x\in X$ 
\[\|H_j(x)\|\leq \|\Phi^{-1}(x)\|_{\sbt} \leq \sup_j\|H_j\|.\] 

Hence $\|H_j\|\leq \|\Phi^{-1}\|\leq \sup_j\|H_j\|$ for all $j\geq 1$, and so $\|\Phi^{-1}\|=\sup_j\|H_j\|$.

In particular, if the scaling factor of $e$ is $1$, then $\Phi$ is an isometry (since $\Phi$ was already contractive). 
Since $\Phi$ is c.p.\ order zero, \cref{a: isometric unitization} tells us it is a complete order embedding. 
Moreover for any $x\in X\backslash\{0\}$, 
\begin{align*}
    \|x\|^4&=\|x^*xx^*x\|\\
    &\leq \|x^*\Phi(x\sbt x^*)x\|\\
    &\leq \|x^*\Phi^\sim(\|x\|_{\sbt}^2 e)x\| \\%\|\Phi(x)^*(\|x\|^2\Phi(1_A))\Phi(x)\|
   & \leq %\|\Phi(x)\|\|x\|^2\|\Phi(1_A)\Phi(x)\|\leq 
    \|x\|\|x\|_{\sbt}^2\|ex\|.
\end{align*}
Since $\|x\|=\|x\|_{\sbt}$, it follows that $\|ex\|=\|x\|$ for all $x\in X$. 
\end{proof}

Now combining \cref{prop 61}, \cref{closed uniform order unit}, and \cref{completely isometric}, we can finish the proof of \cref{theorem B}:

\begin{corollary}\label{corollary E.5}
Let $B$ be a $\Cstar$-algebra and $X\subset B$ a self-adjoint linear subspace. Then $X$ is completely order isomorphic to a unital $\Cstar$-algebra via a c.p.\ order zero map if and only if $X$ is closed in $B$ and there exists an        order unit $e\in B_+^1\cap X\cap X'$ for $X$ with scaling factor $1$ so that $X^2=eX$. 
\end{corollary}

\begin{proof}
Suppose $A$ is a unital $\Cstar$-algebra and $\theta\colon A\longrightarrow  B$ is a completely isometric c.p.\ order zero map with $\theta(A)=X$. Set $e\coloneqq \theta(1_A)$. Then $e\in B_+^1\cap X\cap X'$ is an        order unit for $X$ with scaling factor $1$ such that  $X^2=eX$ by the structure theorem for order zero maps (see \cref{prop52.5}) and \cref{invertible c.p. order zero and order scale}(ii).

On the other hand, if $X$ is closed in $B$ and there exists an        order unit $e\in B_+^1\cap X\cap X'$ for $X$ with scaling factor $1$ so that $X^2=eX$, then by \cref{theorem57}, \cref{prop 61},  and \cref{completely isometric}, $\Phi\coloneqq \id_{X}\colon \Cstar_{\sbt}(X)\longrightarrow  B$ is a complete order embedding from the unital $\Cstar$-algebra $\Cstar_{\sbt}(X)$ with image $X$. 
\end{proof}

When $X$ is closed, we can detect an approximate identity from $\Cstar_{\sbt}(X)$ by how it behaves in $B$ (this essentially underlies \cite[Corollary 2.11]{CW1}). 

\begin{proposition}\label{approx unit}
Let $X\subset B$ be a closed self-adjoint linear subspace with        order unit $e\in B_+^1\cap X'$ such that $X^2\subset eX$, and let  $\Cstar_{\sbt}(X)$ be the associated $\Cstar$-algebra and $C\subset \Cstar_{\sbt}(X)$ a sub-$\Cstar$-algebra. An increasing net $(u_\lambda)_\Lambda\subset C^1_+$ is an approximate unit for $C$ if and only if $\lim_\lambda\|u_\lambda y-ey\|= 0$ for all $y\in C$. 
\end{proposition}

\begin{proof}
Since $X$ is closed, $\Cstar_{\sbt}(X)=X$ as vector spaces, and so $C$ is also a subspace of $B$. From \cref{completely isometric}, we know the scaling factor for $e$ is $\|\Phi^{-1}\|$, and with this \eqref{for plain x} tells us $2\|\Phi^{-1}\|\|ex\|\geq \|x\|$ for any $x\in X$. Now fix $y\in C$ and $\lambda\in \Lambda$. Then $\|u_\lambda y-ey\|=\|e(u_\lambda{\sbt} y-y)\|$, and  
\begin{align*}
   \frac{\|u_\lambda{\sbt} y-y\|_{{\sbt}}}{2\|\Phi^{-1}\|^2}&\leq \frac{\|u_\lambda{\sbt} y-y\|}{2\|\Phi^{-1}\|}\\
   &\leq  \|e(u_\lambda{\sbt} y-y)\|\\ &\leq \|u_\lambda{\sbt} y-y\|\\
   &\leq \|u_\lambda{\sbt} y-y\|_{\sbt},
    \end{align*}
    where the last inequality follows from \cref{equivalent norms}.
The claim follows.
\end{proof}

Finally, 
when our local order unit $e$ lies in $X$, the triple $(X,\{\M_r(X)\cap \M_r(B)_+\}_r,e)$ forms an abstract operator system (see \cite[Chapter 13]{Pau02}), to which we can associate the enveloping $\Cstar$-algebra, denoted $\Cstar_{\text{min}}(X)$ or $\Cstar_{\text{env}}(X)$ (as in \cite{Ham79}). 
When $X$ is moreover closed, 	$\Cstar_{\sbt}(X)$ is $\Cstar_{\text{min}}(X)$.  

\begin{proposition}\label{envelope}
Let $B$ be a $\Cstar$-algebra. \begin{enumerate}[label=\textnormal{(\roman*)}]
\item Suppose $A$ is a unital $\Cstar$-algebra and $\theta\colon A\longrightarrow B$ is isometric and c.p.c.\ order zero. 
Then 
\[\Cstar_{\emph{min}}(\theta(A))\cong A\cong \Cstar_{\sbt}(\theta(A)).\]
\item If $X\subset B$ is a closed self-adjoint linear subspace with        order unit $e\in X\cap X'\cap B_+^1$ with scaling factor $R=1$ such that $X^2=eX$, then $\Cstar_{\sbt}(X)$ is the enveloping $\Cstar$-algebra of $(X,\{\M_r(X)\cap \M_r(B)_+\},e)$. 
\end{enumerate}
\end{proposition}

\begin{proof}
(i) Since $\theta$ is an isomorphic order zero map, it is a complete order embedding by \cref{a: isometric unitization}, and so the first isomorphism follows from the universal property of the minimal $\Cstar$-algebra associated to the abstract operator system $(\theta(A),\{\M_r(X)\cap \M_r(B)_+\}_r,\theta(1_A))$. The second isomorphism follows from \cref{X=Xbullet}. 

Now (ii) follows immediately from (i) and the fact that $\Cstar_{\sbt}(\Phi(\Cstar_{\sbt}(X)))\\ =\Cstar_{\sbt}(X)$.
\end{proof}

\bibliographystyle{plain}
\makeatletter\renewcommand\@biblabel[1]{[#1]}\makeatother
\bibliography{References2}

\end{document}